\documentclass[12pt,twoside,reqno]{amsart}
 \usepackage[reqno]{amsmath} %\usepackage[colorlinks=true,citecolor=blue]{hyperref}
\usepackage{mathptmx,  mathtools,amssymb, amsfonts, amsthm,   enumerate, color, graphicx,   multirow,  enumitem}

\def\bege{\begin{equation}} \def\ende{\end{equation}}
\def\begr{\begin{eqnarray}}   \def\endr{\end{eqnarray}}
\allowdisplaybreaks

\theoremstyle{plain} % just in case the style had changed
\numberwithin{equation}{section}
\newtheorem{theorem}{Theorem}[section]
\newtheorem{proposition}{Proposition}[section]
\newtheorem{corollary}{Corollary}[section]

\newtheorem{remark}{Remark}[section]
\newtheorem{lemma}{Lemma}[section]

\def\D{\mathbb D}\def\dd{\mathrm{d}}\def\daz{\mathrm{d}A(z)}\def\apa{A_{\alpha}^p}\def\za{(1-|z|^2)^{\alpha}}
 \newcommand{\norm}[1]{\lVert#1\rVert}

\begin{document}
	\title[Norm of the generalized Hilbert operator]{Norm of the generalized Hilbert operator  on weighted Bergman spaces }
	
\author{Dongxing Li, Songxiao Li$^\ast$, Weiye Pan,  Hasi Wulan and Mengmeng Zhou }
	
\address{Dongxing Li \\ School of Financial  Mathematics and Statistics, Guangdong University of Finance, Guangzhou 510521, Guangdong, P.R. China. }
    \email{oio211@live.cn}

\address{Songxiao Li\\   Department of Mathematics, Shantou University, Shantou 515063, Guangdong, P.R. China.  }
\email{jyulsx@163.com}

\address{ Weiye Pan  \\ Department of Mathematics, Sichuan University,  Chengdu, Sichuan 610065, P.R. China.}
	\email{pan\_weiye@scu.edu.cn  }

\address{Hasi Wulan  \\ Department of Mathematics, Shantou University, Shantou 515063, Guangdong, P.R. China.}
	\email{wulan@stu.edu.cn  }

\address{Mengmeng Zhou\\   Department of Mathematics, Shantou University, Shantou 515063, Guangdong, P.R. China.  }
\email{25mmzhou@stu.edu.cn}

	
	\begin{abstract}  Several upper bounds   as well as one  lower bound for the operator norm of the generalized Hilbert operator $\mathcal{H}_b$ acting on weighted Bergman spaces $A_{\alpha}^p$ are established.   Moreover, under some mild assumptions,  we obtain the exact norm   of $\mathcal{H}_b$ on $ A_{\alpha}^p$.

	\thanks{$^\ast$ Corresponding author.}
	\end{abstract}
		\maketitle

\section{Introduction} \vskip 2mm

 \subsection{Hilbert matrix and Hilbert's inequality}

 The infinite matrix proposed by David Hilbert in 1894,
\begin{equation*}
\mathbb{H}  \,=\, \begin{pmatrix}
1 & 1/2 & 1/3 & \cdots \\
1/2 & 1/3 & 1/4 & \cdots \\
1/3 & 1/4 & 1/5 & \cdots \\
\vdots & \vdots & \vdots & \ddots
\end{pmatrix} \,=\,\left(\frac{1}{n+k+1}\right)_{n,k=0,1,2,\dots},
\end{equation*}
is referred to as the Hilbert matrix, which is the canonical example of a Hankel matrix.  Recall that an infinite matrix $A=(a_{i,j})$ is a Hankel matrix if
its entries depend only on the sum of its row and column indices,    i.e., $a_{i,j}=a_{i+j}$.

The Hilbert matrix naturally generates a sequence transformation, denoted by $\mathcal{H}$, acting as
\[
\{a_n\}_{n\ge 0} \longmapsto \left\{\sum_{k=0}^\infty\frac{a_k}{n+k+1}\right\}_{n\ge 0},
\]
for all sequences $\{a_n\}$ such that the series on the right-hand side converges absolutely. The classical Hilbert inequality states that
\[
\sum_{m=0}^{\infty} \sum_{n=0}^{\infty}\frac{a_m b_n}{m+n+1}\leq
\frac{\pi}{\sin\frac{\pi}{p}} \left(\sum_{m=0}^{\infty}|a_m|^p\right)^{\frac{1}{p}} \left(\sum_{n=0}^{\infty} |b_n|^q\right)^{\frac{1}{q}},
\]
where the conjugate exponents $p,q\in(1,\infty)$ satisfy $\frac{1}{p}+\frac{1}{q}=1$. Applying this inequality yields
\[
\left( \sum_{i = 0}^{\infty}\left| \sum_{j = 0}^{\infty}\frac{a_{j}}{i + j + 1}\right|^{p}\right)^{1/p} \leq \frac{\pi}{\sin \frac{\pi}{p}} \left( \sum_{k = 0}^{\infty}|a_{k}|^{p}\right)^{1/p}=\frac{\pi}{\sin\frac{\pi}{p}} \left\| \{a_n\} \right\|_{\ell^p}.
\]
This bound directly implies that the operator $\mathcal{H}: \ell^p \to \ell^p$ is bounded for every   $ p\in (1,\infty)$.

 \subsection{Known results of Hilbert operator on Hardy spaces and Bergman spaces}

 Let $\mathbb{D}$ denote the open unit disk of the complex plane $\mathbb{C}$, and let $\mathrm{H}(\mathbb{D})$ stand for the   space of all   analytic functions on $\mathbb{D}$.
Take any $f\in\mathrm{H}(\mathbb{D})$ with Taylor expansion $f(z)=\sum_{n=0}^{\infty}a_{n}z^{n}$. The Hilbert matrix $\mathbb{H}$ acts formally on the coefficient sequence $(a_{n})_{n\ge 0}$, which yields the formal power series representation
\[
\mathcal{H}(f)(z)=\sum_{n=0}^{\infty}\left(\sum_{k=0}^{\infty}\frac{a_{k}}{n+k+1}\right)z^{n}.
\]
For the constant function $f(z)=\sum_{n=0}^{\infty}z^{n}$, the inner series $\sum_{k=0}^{\infty}\frac{1}{n+k+1}$ diverges to $+\infty$ for every nonnegative integer $n$. Consequently, the mapping $\mathcal{H}$ cannot be defined as an operator on the entire space $\mathrm{H}(\mathbb{D})$.

In 2000, Diamantopoulos and Siskakis \cite{DS} initiated systematic research on the boundedness of the aforementioned operator $\mathcal{H}$, hereafter referred to as the Hilbert operator, acting on Hardy spaces $H^{p}$. They verified that $\mathcal{H}(f)$ is analytic on $\mathbb{D}$ for every $f\in H^{1}$, and established the integral representation identity
\[
\mathcal{H}(f)(z)=\sum_{n=0}^{\infty}\sum_{k=0}^{\infty}\frac{a_{k}}{n+k+1}z^{n}=\int_{0}^{1}\frac{f(t)}{1-tz}\,\mathrm{d}t
\]
valid for all $f\in H^{1}$. Their main results show that $\mathcal{H}$ fails to be bounded on $H^{1}$ and $H^{\infty}$, while $\mathcal{H}:H^{p}\to H^{p}$ is a bounded linear operator whenever $1<p<\infty$.  In 2008, Dostanić, Jevtić and Vukotić \cite{DJV} sharpened the preceding results by computing the exact operator norm of $\mathcal{H}$ on $H^{p}$ as follows.
\[
\|\mathcal{H}\|_{H^{p}\to H^{p}}=\frac{\pi}{\sin\frac{\pi}{p}}.
\]

In \cite{Di}, Diamantopoulos proved that the  operator \(\mathcal{H}\) acts boundedly on the Bergman space \(A^p\) if and only if   \(p > 2\); meanwhile, a sharp upper estimate of the operator norm \(\|\mathcal{H}\|_{A^p \rightarrow A^p}\) was established when \(p \geq 4\). By constructing suitable test functions, Dostanić, Jevtić and Vukotić  \cite{DJV} deduced an exact lower bound of \(\|\mathcal{H}\|_{A^p \rightarrow A^p}\) for all \(p > 2\), which further yields the exact norm value for \(p \geq 4\). For   \(2 < p < 4\),   Bozin and Karapetrović \cite{BK} obtained the corresponding sharp upper bound. Specifically, it was shown that,  for each \(p > 2\),
\[
\|\mathcal{H}\|_{A^p \rightarrow A^p} = \frac{\pi}{\sin \frac{2\pi}{p} }  .
\]

\subsection{Known results of Hilbert operator on weighted Bergman spaces}

 It is known from \cite{JK} that \(\mathcal{H}\) is bounded on the weighted Bergman space \(A^p_\alpha\) if and only if   \(1 < \alpha + 2 < p\).    Later, Karapetrović in \cite{Bo} deduced a lower norm estimate and a family of upper norm estimates for the operator $\mathcal{H}$ acting on $A_\alpha^p$. Specifically, Karapetrović proved the equality
\begin{align} \label{tkcj}
\|\mathcal{H}\|_{A_\alpha^p \rightarrow A_\alpha^p } = \frac{\pi}{\sin \frac{(\alpha + 2)\pi}{p}}
\end{align}
when $4 \leq 2(\alpha + 2) \leq p$. Based on the aforementioned result, Karapetrovi\'c formulated the following conjecture: for any parameters satisfying $\alpha > -1$ and $p > \alpha + 2$, the equality in \eqref{tkcj} remains valid. The table below lists the main contributions toward this conjecture in recent years.

\renewcommand{\arraystretch}{1.5}
\begin{tabular}{|p{5cm}|p{11cm}|}
  \hline
  \textbf{Authors} & \textbf{Conditions for the conjecture to hold} \\[1mm]
  \hline
  Karapetrović \cite{Bo} & $\alpha \geq 0$ and $p \geq 2(\alpha + 2)$ \\[1mm]
  \hline
  Lindstr\"{o}m, Mihkinen, and Wikman \cite{LM} &
  $\alpha > 0$ and $\alpha + 2 + \sqrt{\alpha^2 + \frac{7}{2}\alpha + 3} \leq p < 2(\alpha + 2)$ \\[1mm]
  \hline
  \multirow{2}{*}{Karapetrović \cite{ka}} &
  $\alpha > 0$ and $\alpha + 2 + \sqrt{(\alpha+2)^2 -(\sqrt{2}-\frac{1}{2})(\alpha + 2)} \leq p < 2(\alpha + 2)$ \\
  \cline{2-2}
  & $\alpha > 0$ and $\alpha_0 \leq p < 2(\alpha + 2)$, see \cite{ka} for the detail of $\alpha_0$. \\[1mm]
  \hline
  \multirow{4}{*}{Dai \cite{Da}} &
  $\alpha = 1$ \\
  \cline{2-2}
  & $0 < \alpha \leq \frac{1}{47}$ and $p > \alpha + 2$ \\
  \cline{2-2}
  & $\alpha > 0$ and $2 + \frac{3\alpha}{4} + \frac{1}{4}\sqrt{9\alpha^2 + 40\alpha + 48} \leq p < 2(\alpha + 2)$ \\
  \cline{2-2}
  & $-1 < \alpha < 0$ and $p \geq 2(\alpha + 2)$ \\[1mm]
  \hline
   Bao, Tian and Wulan \cite{BTH}& $\alpha \ge 0$ and $ \frac{9+3\alpha+\sqrt{9\alpha^{2}+30\alpha+33}}{4} \le p<2(\alpha+2)$ \\[1mm]
  \hline
\end{tabular}\\

Recently, Wulan, Zhou and Zhu \cite{wzz} proved that, for even exponents $p=2m$,
\[
 \norm{\mathcal{H}}_{A^{2m}_\alpha\to A^{2m}_\alpha}
 =B(a,1-a),\qquad a=\frac{\alpha+2}{2m},
\]
subject to the condition $0<a\le m/(2m-1)$. This norm equality holds for all admissible parameters when $p=2,4,6,8,10$. However, Karapetrović's conjecture fail for every even integer $p=2m$ with $m\ge550000$.

 \subsection{Generalized Hilbert operator on Hardy spaces}
 Let $b\geq 0$. Let \( \mathbb{N}_0 \) denote  the set of all non-negative integers.    Li and Stevi\'c \cite{ls} introduced the generalized Hilbert matrix as follows.
\begin{equation*}
\begingroup
\renewcommand{\arraystretch}{1.4}  % 增大行间距，默认是1
\mathbb{H}_b=\left( \frac{\Gamma(n + b + 1) \Gamma(n + k + 1)}{\Gamma(n + 1) \Gamma(n + k + b + 2)} \right)_{n,k \in \mathbb{N}_0}=\begin{pmatrix}
 \frac{1}{b+1} &  \frac{1}{(b+1)(b+2)} &  \frac{2}{(b+1)(b+2)(b+3)} & \cdots\\
 \frac{1}{b+2} & \frac{2}{(b+2)(b+3)} & \frac{6}{(b+2)(b+3)(b+4)} & \cdots\\
 \frac{1}{b+3} & \frac{3}{(b+3)(b+4)} &  \frac{12}{(b+3)(b+4)(b+5)}& \cdots\\
 \vdots & \vdots & \vdots & \ddots
\end{pmatrix}  .
\endgroup
\end{equation*}
It is easy to check that
  \begin{equation*}
\mathbb{H}_b=
\begingroup
\renewcommand{\arraystretch}{1.4}  % 增大行间距，默认是1
\begin{pmatrix}
 b! & 0 & 0 & \cdots \\
0 & (b+1)! & 0 & \cdots \\
 0 & 0 & \frac{(b+2)!}{2} & \cdots \\
 \vdots & \vdots & \vdots & \ddots
\end{pmatrix}
\endgroup
\begingroup
\renewcommand{\arraystretch}{1.4}  % 增大行间距，默认是1
\begin{pmatrix}
 \frac{1}{(b+1)!} & \frac{1}{(b+2)!} & \frac{2}{(b+3)!} & \cdots \\
 \frac{1}{(b+2)!} & \frac{2}{(b+3)!} & \frac{6}{(b+4)!} & \cdots \\
 \frac{2}{(b+3)!} & \frac{6}{(b+4)!}& \frac{24}{(b+5)!} & \cdots \\
 \vdots & \vdots & \vdots & \ddots
\end{pmatrix}.
\endgroup
\end{equation*}
Hence, the generalized Hilbert matrix can be decomposed into the product of a diagonal matrix and a Hankel matrix.

This generalized Hilbert matrix also induces a bounded linear operator on $\mathrm{H}(\mathbb{D})$, termed the generalized Hilbert operator and denoted by $\mathcal{H}_b$. For any analytic function $f(z) = \sum_{k=0}^{\infty} a_k z^k$, the operator acts on the Taylor coefficients of $f$ via the mapping
\begin{align*}
a_n \mapsto \sum_{k=0}^{\infty} \frac{ \Gamma(n + b + 1) \Gamma(n + k + 1)}{\Gamma(n + 1) \Gamma(n + k + b + 2)}   a_k, \quad n = 0, 1, 2, \cdots.
\end{align*}
Li and Stević \cite{ls} established some norm estimates for the operator $\mathcal{H}_b$ acting on Hardy spaces $H^p$. Additionally, the authors proposed an unresolved conjecture stating that the equality
\[
\|\mathcal{H}_b\|_{H^{p}\to H^{p}} = B\Big(\frac{1}{p},~b+1-\frac{1}{p}\Big)
\]
holds for  $b \geq 0$ and $1 < p < \infty$. See \cite{bw, li} for more results about generalized Hilbert operator  $\mathcal{H}_b$.

\subsection{Main results and an open problem}

Motivated by the aforementioned conclusions, this paper gives multiple upper norm bounds and one lower norm bound for the generalized Hilbert operator $\mathcal{H}_b$ acting on weighted Bergman spaces $A_{\alpha}^p$, where the weight index $\alpha$ satisfies $\alpha\geq0$ and $-1<\alpha<0$. With appropriate constraints imposed on the involved parameters, we obtain the exact operator norm equality
\begin{align} \label{t118899}
\|\mathcal{H}_b\|_{A_{\alpha}^p\rightarrow A_{\alpha}^p}= B\Big(\frac{\alpha+2}{p},~b+1-\frac{\alpha+2}{p}\Big).
\end{align}
This identity is valid for all $\alpha > -1$ and $p\geq2(\alpha+2)$; see Corollary \ref{cor1} and Corollary \ref{cor4.1} for details. The main contributions of the present paper can be   summarized into three aspects, as listed below.
\begin{enumerate}

\item Our theorems adopt weaker hypotheses compared with the common assumptions adopted in existing literature concerning Hilbert operators. Specifically, for Case $\bf{(ii)}$ of Theorem \ref{thm3}, the admissible interval of parameter $p$ extends beyond the valid scope derived from classical results (see \cite[Theorem 1.2]{Bo}). Meanwhile, the characterization presented in Theorem \ref{thm7} possesses finer structural properties relative to the counterpart conclusion stated in \cite[Theorem 1.3]{ka}. Furthermore, the established results contain several classical theorems as special cases.

\item Our proofs rely on integral representations, backward shift operators, identities of the Beta function, and fundamental properties of hypergeometric functions. Furthermore, the arguments for Theorem \ref{thm3} $(\mathbf{v})$ and Theorem \ref{thm7}$(\mathbf{iv})$ differ essentially from the treatment of the classical Hilbert operator $\mathcal{H}$,  where novel estimation techniques and structural constructions are incorporated within the   main steps of the proof.

\item  Theorem \ref{thm3} together with Theorem \ref{thm7} indicates that the generalized operator $\mathcal{H}_b$ is more than a parametric generalization of $\mathcal{H}$. The parameter $b$ imposes essential structural modifications on the operator itself.

\end{enumerate}

To conclude this section, we  formulate  an open problem   as follows.

  \textbf{Open Problem.}  Determine the full range of parameters such that the equality (\ref{t118899})
holds true.

Throughout this paper, we  adopt standard notation as follows: the relation $f \lesssim g$ signifies that there exists a positive constant $C$ satisfying $f \leq Cg$, while $f \asymp g$ holds if and only if both $f \lesssim g$ and $g \lesssim f$ are valid.

\section{Preliminary}

In this section, we give some definitions  and preliminary results, which will used in the proof of main results in this paper.

\subsection{Hardy and weighted Bergman spaces}

 For \( 0 < p \leq \infty \), the Hardy space \( H^p \) consists of those functions \( f \in \mathrm{H}(\mathbb{D}) \) for which
\begin{equation*}
\|f\|_{H^p} = \sup_{0 < r < 1} M_p(r, f) < \infty.
\end{equation*} Here  $M_{\infty}(r, f) = \sup_{|z|=r} |f(z)|$ and
\begin{align*}
M_p(r, f) = \left( \frac{1}{2\pi} \int_{-\pi}^{\pi} |f(re^{i\theta})|^p \dd\theta \right)^{\frac{1}{p}}, \quad 0 < p < \infty.
\end{align*}

Let \( 0 < p < \infty \)  and \( -1 < \beta < \infty \). The weighted Bergman space \( A^p_\beta  \) is the space of all \( f \in \mathrm{H}(\mathbb{D}) \) such that
\[
\|f\|^p_{ A^p_\beta } =  \int_{\mathbb{D}} |f(z)|^p \, \dd A_\beta(z) < \infty,
\]
where \(\dd A(z) \) is the normalized area measure on \( \mathbb{D} \) and \( \dd A_\beta(z) = (\beta + 1)(1 - |z|^2)^\beta \, \dd A(z) \).
One easily verifies that $f \in A_\alpha^p$ if and only if
$
 \int_{0}^{1}(1-r)^{\alpha}M_{p}^{p}(r,f)\mathrm{d}r<\infty.
$

\subsection{Beta   and hypergeometric function}  Recall  the Beta function $B(s,~t)$ is defined by
\[
B(s,~t) = \int_0^1 x^{s-1}(1-x)^{t-1}\mathrm{d}x = \int_0^\infty \frac{x^{s-1}}{(x+1)^{s+t}}\mathrm{d}x,
\]
which is valid for all complex parameters $s,t$ satisfying $\operatorname{Re} s > 0$ and $\operatorname{Re} t > 0$. It is well known that
\[
B(s,~t) = \frac{\Gamma(s)\Gamma(t)}{\Gamma(s+t)},
\]
where the Gamma function $\Gamma(z)$ is defined by
\[
\Gamma(z) = \int_0^\infty t^{z-1}e^{-t}\mathrm{d}t = \int_0^1 \Big(\ln \frac{1}{x}\Big)^{z-1} \mathrm{d}x,\quad \operatorname{Re} z > 0.
\]

The Gauss hypergeometric series is defined as
\[
F(a,b,c;z)=\sum_{k=0}^{\infty}\frac{\Gamma(k+a)\Gamma(k+b)\Gamma(c)}{\Gamma(a)\Gamma(b)\Gamma(k+c)}\cdot\frac{z^k}{k!},\quad z\in\mathbb{D}.
\]
This hypergeometric function  has the following  integral representation (see \cite{Bo})
\begin{align}\label{eq8}
F(a,b,c;z)=\frac{1}{B(a,c-a)}\int_0^1\frac{t^{a-1}(1-t)^{c-a-1}}{(1-tz)^b}\mathrm{d}t,\quad \operatorname{Re} c>\operatorname{Re} a>0.
\end{align}

\begin{lemma}\label{lem2.1}
    Let $0 < x < b+1$ and $y > 0$. Then
\begin{align*}
B(x+y, ~b+1-x) = \frac{B(x,~b+1-x)B(y,~b+1)}{ B(x,~y)}.
\end{align*}
  \end{lemma}

  \begin{proof}   By the properties of the Beta function and Gamma function, we have
\begin{align*}
B(x+y, ~b+1-x) &= \frac{\Gamma(x+y)\Gamma(b+1-x)}{\Gamma(y+b+1)}  = \frac{\Gamma(x+y)\Gamma(b+1-x)\Gamma(x)\Gamma(b+1)\Gamma(y)}{\Gamma(y+b+1)\Gamma(x)\Gamma(b+1)\Gamma(y)} \\[0.1cm]
&=\frac{B(x,~ b+1-x)B(y, ~b+1)}{ B(x,~ y)},
\end{align*}
as desired.
  \end{proof}

  \begin{lemma}[\cite{Bh}]\label{lem2.2}
    Let $0 < x \leq 1$. Then
\begin{align*}
\frac{x + y - xy}{xy} \leq B(x,~ y) \leq \frac{x + y}{xy(1 + xy)}
\end{align*}
for all $0 < y \leq 1$, and
\begin{align*}
\frac{x + y}{xy(1 + xy)} \leq B(x,~ y) \leq \frac{x + y - xy}{xy}
\end{align*}
for all $y > 1$.
  \end{lemma}

   \begin{lemma}[\cite{liu}]\label{l223}
For \( a \in \mathbb{R} \) and \( t > -1 \), we have
\[
\int_{\mathbb{D}} \frac{(1 - |\zeta|^2)^t}{|1 -  \bar{ \zeta} \eta  |^{2a}}\mathrm{d}  A(\zeta )
= \frac{  \Gamma(1+t)}{\Gamma(2+t)}
\, F(a,a, 2+t; |\eta|^2)
\]
holds for all \( \eta \in \mathbb{D}\).
  \end{lemma}

 \subsection{Boundedness of generalized Hilbert operator on weighted Bergman spaces} Let $\alpha>-1$, $b\geq 0$, $1< p< \infty$ and  $f \in A_{\alpha}^{p}$.
 Define the integral operator $\mathcal{T}_b$ via
\begin{equation} \nonumber
\mathcal{T}_{b}(f)(z)=\int_{0}^{1}\frac{f(t)(1-t)^{b}}{(1-tz)^{b+1}}\,\mathrm{d}t, ~~z\in  \mathbb{D}.
\end{equation}
Under the condition $ \alpha+2<p(b+1)$,   from \cite[Theorem 4.14]{Zhu2} we get
\begin{align*}
\left | \mathcal{T}_{b}(f)(z)  \right |  \le \int_{0}^{1}\frac{\left | f(t) \right |(1-t)^{b} }{\left | 1-tz \right |^{b+1} } \,\mathrm{d}t \lesssim \frac{1}{(1-\left | z \right |)^{b+1}}\left(\int_{0}^{1}\frac{1}{(1-t)^{\frac{\alpha+2}{p}-b}}\,\mathrm{d}t\right)\left \|  f\right \|_{A_{\alpha}^{p}}<+\infty.
\end{align*}

Take any $f(z)=\sum_{n=0}^{\infty}a_{n}z^{n}\in A_{\alpha}^{p}$ and denote its $N$-th partial sum by $f_{N}(z)=\sum_{n=0}^{N}a_{n}z^{n}$. Direct computation yields
\begin{align*}
\mathcal{H}_{b}(f_{N})(z)&=\sum_{n=0}^{\infty} \left( \sum_{k=0}^{N}\frac{\Gamma(n+b+1)\Gamma(n+k+1)}{\Gamma(n+1)\Gamma(n+k+b+2)} a_{k}\right)z^{n}\\[0.1cm]
&=\sum_{n=0}^{\infty} \frac{\Gamma(n+b+1)}{\Gamma(n+1)\Gamma(b+1)}\left( \sum_{k=0}^{N}\left(\int_{0}^{1}(1-t)^{b}t^{n+k}\,\mathrm{d}t\right) a_{k} \right) z^{n}\\[0.1cm]
&=\sum_{n=0}^{\infty} \frac{\Gamma(n+b+1)}{\Gamma(n+1)\Gamma(b+1)}\int_{0}^{1}(1-t)^{b}f_{N}(t)(tz)^{n}\,\mathrm{d}t \\[0.1cm]
&=\mathcal{T}_{b}(f_{N})(z).
\end{align*}
The identity above implies that $\mathcal{H}_{b}$ is well defined on the set of all analytic polynomials. For arbitrary $z\in\mathbb{D}$ and $ \alpha+2<p(b+1)$, we have
\begin{align*}
&\left | \mathcal{T}_{b}(f)(z)- \sum_{n=0}^{\infty} \left( \sum_{k=0}^{N}\frac{\Gamma(n+b+1)\Gamma(n+k+1)}{\Gamma(n+1)\Gamma(n+k+b+2)} a_{k}\right)z^{n} \right |\\[0.1cm]
\lesssim &  \frac{1}{(1-\left | z \right | )^{b+1}}\int_{0}^{1}\left | f(t)-f_{N}(t) \right |(1-t)^{b}\,\mathrm{d}t \\[0.1cm]
\lesssim &\frac{1}{(1-\left | z \right | )^{b+1}}\left(\int_{0}^{1}\frac{1}{(1-t)^{\frac{\alpha+2}{p}-b}}\,\mathrm{d}t\right)\left \| f-f_{N} \right \|_{A_{\alpha}^{p}} .
\end{align*}
As $N \to \infty$, the partial series
$
\sum_{n=0}^{\infty} \left( \sum_{k=0}^{N}\frac{\Gamma(n+b+1)\Gamma(n+k+1)}{\Gamma(n+1)\Gamma(n+k+b+2)} a_{k}\right)z^{n}
$
converges pointwise on $\mathbb{D}$. Moreover,
\begin{align}\nonumber
  \mathcal{H}_{b}(f)(z)=\mathcal{T}_{b}(f)(z)=\int_{0}^{1}\frac{f(t)(1-t)^{b}}{(1-tz)^{b+1}}\,\mathrm{d}t,
\end{align}
 holds for all $f\in A_{\alpha}^{p}$ subject to the condition $ \alpha+2<p(b+1)$.

 \begin{lemma}\cite[Lemma 3]{ggps} \label{lem211}
 Assume that $0<p<\infty$ and $\alpha>-1$. Then there exists a positive constant $C=C(p,\alpha)$ such that
 $$
 \int_0^1 M^p_\infty(r,f)(1-r)^{\alpha+1}\,dr  \le C ||f||^p_{A^p_\alpha},  \quad\text{for all} ~~~ f\in \mathrm{H}(\mathbb{D}) .
 $$
 \end{lemma}

Next, we give a characterization of the boundedness of $\mathcal{H}_{b}$ on $\apa$. Specifically, a necessary and sufficient condition is provided.

\begin{proposition}
Let $\alpha>-1$, $b\geq 0$ and $1< p< \infty$. Then $\mathcal{H}_{b}$ is bounded on $\apa$ if and only if $$\alpha+2<p(b+1).$$
\end{proposition}

\begin{proof} We first assume the operator $\mathcal{H}_{b}$ is bounded on $A_{\alpha}^{p}$. Assume the contrary that $p(b+1) \le \alpha+2$, and take an arbitrary positive real number $\delta>0$. Define
$$
f(z) = (1-z)^{-\frac{\alpha+2}{p} + \delta}.
$$
It is straightforward to verify $f\in A^p_\alpha$. Direct computation yields
\begin{align} \label{ttt}
\mathcal{H}_b f(r) = \int_0^1 (1-s)^{b - \frac{\alpha+2}{p} + \delta} (1-sr)^{-(b+1)} \mathrm{d} s.
\end{align}
 By Proposition 4.13 of \cite{Zhu2}, we have
\begin{align*}
 \left | \mathcal{H}_b f(r) \right|\cdot (1-r^{2})^{\frac{\alpha+2}{p}} \lesssim \left \| \mathcal{H}_b f \right \|_{A_{\alpha}^{p}}
\end{align*}
holds for all $r\in(0,1)$. Substitute $r=\frac{1}{2}$ into (\ref{ttt}), then we obtain
\begin{align*}
\left \| \mathcal{H}_b f \right \|_{A_{\alpha}^{p}} \gtrsim  \left |  \mathcal{H}_b f (\tfrac12 ) \right|\cdot \left(1-\left(\tfrac12\right)^{2}\right)^{\frac{\alpha+2}{p}} \asymp \int_{0}^{1}(1-s)^{b-\frac{\alpha+2}{p}+\delta }\mathrm{d} s.
\end{align*}
The improper integral on the right-hand side diverges provided $b+1\le \frac{\alpha+2}{p}-\delta$. Since $\delta>0$ is chosen arbitrarily, the inequality $b+1\le \frac{\alpha+2}{p}$ necessarily leads to $$\left \| \mathcal{H}_b(f) \right \|_{A_{\alpha}^{p}}\to \infty.$$
 This contradicts the assumed boundedness of $\mathcal{H}_{b}$ acting on $A_{\alpha}^{p}$. Therefore, the boundedness of $\mathcal{H}_{b}$ on $A_{\alpha}^{p}$ enforces the strict inequality $p(b+1)>\alpha+2$.

Conversely, suppose $ \alpha+2<p(b+1)$. Let $f\in A^p_\alpha$.  Using Minkowski's inequality and the following well-known inequality
\[
\int_{0}^{2\pi} \frac{\dd\theta}{|1 - t e^{i\theta}|^s}
= O\!\Big(\frac{1}{(1-t)^{s-1}}\Big), \qquad s>1,
\]
 we get
\begin{align*}
M_{p}(r, \mathcal{H}_b(f) )&= \left ( \frac{1}{2 \pi}\int_{0}^{2 \pi}\left | \int_0^1f(t)\frac{(1-t)^b }{(1-tre^{i \theta})^{b+1}}\dd t \right |^{p}\dd \theta  \right )^{\frac{1}{p}}\nonumber \\[0.1cm]
  & \lesssim \int_{0}^{1}(1-t)^{b}\left | f(t) \right |\left ( \int_{0}^{2 \pi}\frac{1}{\left | 1-tr e^{i \theta} \right |^{p(b+1)}} \dd \theta \right )^{\frac{1}{p}}\dd t\nonumber \\[0.1cm]
  & \lesssim \int_{0}^{1}(1-t)^{b}\left | f(t) \right |\left ( \frac{1}{(1-tr)^{p(b+1)-1}} \right )^{\frac{1}{p}} \dd t\nonumber \\[0.1cm]
%  & =\int_{0}^{1}(1-t)^{b}\left | f(t) \right |\frac{1}{(1-tr)^{b+1-\frac{1}{p}}}  \dd t\nonumber \\[0.1cm]
  & \le \int_{0}^{1}(1-t)^{b}M_{\infty}(t,f)\frac{1}{(1-tr)^{b+1-\frac{1}{p}}}  \dd t.
\end{align*}
By Lemma \ref{lem211} and \cite[Proposition 7.3.2]{JVA}, we have
\begin{align*}\nonumber
\left \|   \mathcal{H}_{b}(f)  \right \|_{A^p_\alpha}^{p}&\asymp\int_{0}^{1}M_{p}^{p}\left ( r,  \mathcal{H}_{b}(f)   \right )\left ( 1-r \right )^{  \alpha } \dd r\nonumber \\[0.1cm]
&\lesssim \int_{0}^{1}\left ( \int_{0}^{1}(1-t)^{b}M_{\infty}(t,f)\frac{1}{(1-tr)^{b+1-\frac{1}{p}}}  \dd t \right )^{p} \left ( 1-r \right )^{\alpha} \dd r\nonumber \\[0.1cm]
&\lesssim \int_{0}^{1}M_{\infty}^{p}(r,f)\left ( 1-r \right )^{\alpha+1}\dd r  \lesssim   \left \|  f \right \|_{A^p_\alpha}^{p}<\infty.
\end{align*}
Consequently,  $\mathcal{H}_{b}$ is bounded on $\apa$. The proof is complete.
\end{proof}

\subsection{Some lemmas}

In this subsection, we collect and prove some lemmas, which will be used in the rest of this paper.

  \begin{lemma}[\cite{Da}]\label{lem308-5}
    Let $f(x)$ be nonnegative, continuously differentiable and monotonically increasing on the interval $[a,b)$. Suppose that $g(x)$ is continuous and integrable on $(a,b)$ satisfying
\begin{align*}
\int_a^b g(x) \dd x \leq 0.
\end{align*}
If $f(x)g(x)$ is integrable on $(a,b)$ and there exists some $c \in (a,b)$ such that $g(x) > 0$ for $x \in (a,c)$ and $g(x) \leq 0$ for $x \in [c,b)$, then
\begin{align*}
\int_a^b f(x)g(x) \dd x \leq 0.
\end{align*}
  \end{lemma}

   In the rest of this paper, we set
  \begin{align}\label{eq307-1}
    \psi_{b,p,\alpha}(s)=s^{\frac{ \alpha+2}{p}-1}(1-s)^{b-\frac{\alpha+2}{p}},\,\,\,s\in(0,1),
  \end{align}
and  $$
p_\alpha=
\dfrac{3\alpha}{4}+2+\frac{1}{4}\sqrt{{9\alpha^2+40\alpha+48}}.
$$
It should be note that
  \begin{align*}
    \int_0^1\psi_{b,p,\alpha}(s)\dd s=B\Big(\frac{\alpha+2}{p},~b+1-\frac{\alpha+2}{p}\Big).
  \end{align*}

\begin{lemma}\label{lem308-6}  Let $\alpha \geq 0$, $b\geq 0$, $1< p< \infty$  such that  $\alpha+2<p(b+1)$. Define
\begin{align}\label{eq324.0}
f(s) = s^{2p-3\alpha-8} \int_{0}^{s} \psi_{b,p,\alpha}(t) \dd t - \int_{0}^{1} \psi_{b,p,\alpha}(t) \dd t, \qquad s \in (0,1].
\end{align}
\begin{enumerate}
  \item [{\bf(i)}] If $p \geq p_{\alpha}$ and  $b\leq\frac{\alpha + 2}{p}$, then $f(s) \leq 0$ for all $s$ in $(0,1]$.
  \item [{\bf(ii)}]  If $p < p_{\alpha}$ and $b<\frac{\alpha + 2}{p}$, then there exists some constant $c \in (0,1)$ such that $f(s) > 0$ for $s \in (0,c)$ and $f(s) \leq 0$ for $s \in [c,1]$.
  \item [{\bf(iii)}] If $p \geq p_{\alpha}$, $2p - 3\alpha - 8 \geq 0$ and $b>\frac{\alpha + 2}{p}$, then $f(s) \leq 0$ for all $s$ in $(0,1]$.
\end{enumerate}

\end{lemma}
\begin{proof}
 Combining (\ref{eq307-1}) and differentiating both sides of \eqref{eq324.0} simultaneously yield
\begin{align*}
f'(s)& = (2p - 3\alpha - 8)s^{2p-3\alpha-9} \int_{0}^{s} \psi_{b,p,\alpha}(t)\dd t + s^{2p-3\alpha-8} \psi_{b,p,\alpha}(s) = s^{2p-3\alpha-9}g(s),
\end{align*}
where
\begin{align*}
g(s) &= (2p - 3\alpha - 8) \int_0^s \psi_{b,p,\alpha}(t) \dd t + s \psi_{b,p,\alpha}(s).
\end{align*}
% Note that
%\begin{align*}
%  s\psi'_{b,p,\alpha}(s)&=s\left( s^{\frac{\alpha + 2}{p} - 1} (1 - s)^{b-\frac{\alpha + 2}{p}}\right)'\\[0.1cm]
%  &=\Big(\frac{\alpha + 2}{p} - 1\Big)s^{\frac{\alpha + 2}{p} - 1} (1 - s)^{b-\frac{\alpha + 2}{p}}+\Big(\frac{\alpha + 2}{p}-b\Big)s^{\frac{\alpha + 2}{p} }(1 - s)^{b-\frac{\alpha + 2}{p}-1}.
%\end{align*}
We obtain
\begin{align*}
g'(s) &= (2p - 3\alpha - 8) \psi_{b,p,\alpha}(s) +\psi_{b,p,\alpha}(s)+ s \psi'_{b,p,\alpha}(s) \\[0.1cm]
%&= (2p - 3\alpha - 7) \psi_{b,p,\alpha}(s) + s \psi'_{b,p,\alpha}(s) \\[0.1cm]
&= (2p - 3\alpha - 7) s^{\frac{\alpha + 2}{p} - 1} (1 - s)^{b-\frac{\alpha + 2}{p}}  \\[0.1cm]
&\quad+ \Big(\frac{\alpha + 2}{p} - 1\Big)s^{\frac{\alpha + 2}{p} - 1} (1 - s)^{b-\frac{\alpha + 2}{p}}+\Big(\frac{\alpha + 2}{p}-b\Big)s^{\frac{\alpha + 2}{p} }(1 - s)^{b-\frac{\alpha + 2}{p}-1}\\[0.1cm]
%&= s^{\frac{\alpha + 2}{p} - 1} (1 - s)^{b-\frac{\alpha + 2}{p}-1}\Big[(2p - 3\alpha - 7 )(1-s)+ \Big(\frac{\alpha + 2}{p} - 1\Big)(1-s)+ \Big(\frac{\alpha + 2}{p} -b\Big)s\Big]\\[0.1cm]
%&= s^{\frac{\alpha + 2}{p} - 1} (1 - s)^{b-\frac{\alpha + 2}{p}-1}\Big[\Big(2p - 3\alpha - 8+ \frac{\alpha + 2}{p}\Big)(1-s)+\Big(\frac{\alpha + 2}{p} -b\Big)s\Big]\\[0.1cm]
&= s^{\frac{\alpha + 2}{p} - 1} (1 - s)^{b-\frac{\alpha + 2}{p}-1}\Big[\left(8+3\alpha - 2p-b\right)s+2p - 3\alpha+ \frac{\alpha + 2}{p} - 8\Big]\\[0.1cm]
&:= s^{\frac{\alpha + 2}{p} - 1} (1 - s)^{b-\frac{\alpha + 2}{p} - 1} h(s),
\end{align*}
where
\begin{align*}
h(s) &= \left(8+3\alpha - 2p-b\right)s+2p - 3\alpha+ \frac{\alpha + 2}{p}- 8.
\end{align*}

{\bf Case~(i)~}. Since $b\leq\frac{\alpha + 2}{p}$ and $p \geq p_\alpha$ or equivalently $2p - 3\alpha + \frac{\alpha + 2}{p} - 8 \geq 0$, for $s\in[0,1]$ we must have
\begin{align*}
h(s) &\geq \min\{h(0), h(1)\} = \min\big\{2p - 3\alpha + \frac{\alpha + 2}{p} - 8, \frac{\alpha + 2}{p}-b\big\} \geq 0.
\end{align*}
Thus, we have $g'(s) \geq 0$ and $g(s)$ is increasing on $(0,1)$. Note that $\lim\limits_{s \rightarrow 0^{+}} g(s) =0$, we deduce that $g(s)$ and $f'(s)$ are nonnegative on $(0,1)$. Therefore, $f(s)$ is increasing on $(0,1]$, so that $f(s) \leq f(1) = 0$ for all $s$ in $(0,1]$.

 {\bf Case (ii)}. Assume that $b<\frac{\alpha + 2}{p}$ and $p < p_\alpha$, that is $2p - 3\alpha + \frac{\alpha + 2}{p} - 8 < 0$. A direct calculation gives that
$h(0) < 0 ~~\text{and} ~~ h(1) > 0.$ This implies the existence of some $s_0 \in (0,1)$ satisfying
\begin{align*}
h(s) &\leq 0 \quad \text{for} \quad s \in [0, s_0] \quad \text{and} \quad h(s) \geq 0 \quad \text{for} \quad s \in [s_0, 1].
\end{align*}
As a consequence,
\begin{align*}
g'(s) &\leq 0 \quad \text{for} \quad s \in (0, s_0] \quad \text{and} \quad g'(s) \geq 0 \quad \text{for} \quad s \in [s_0, 1).
\end{align*}
Thus, $g(s)$ is non-increasing on $(0, s_0]$ and non-decreasing on $[s_0, 1)$. Since $b<\frac{\alpha + 2}{p}$, it can be readily verified that $\lim\limits_{s \rightarrow 1^{-}} g(s) =+\infty$. In fact,
 \begin{align*}
  \lim_{s \to 1^-}g(s)&=  \lim_{s \to 1^-}(2p - 3\alpha - 8) \int_0^s \psi_{b,p,\alpha}(t) \dd t +   \lim_{s \to 1^-}s \psi_{b,p,\alpha}(s)\\[0.1cm]
  &=(2p - 3\alpha - 8) \int_0^1 \psi_{b,p,\alpha}(t) \dd t +   \lim_{s \to 1^-}s^{\frac{\alpha+2}{p}}(1-s)^{b-\frac{\alpha+2}{p}}\\[0.1cm]
  &=+\infty.
\end{align*}
Combining with $\lim\limits_{s \rightarrow 0^{+}} g(s) =0$ , it can be deduced  that there exists $s_1 \in (s_0, 1)$ such that
$f'(s)   \leq 0 $ for $s \in (0, s_1]$ and $f'(s) \geq 0 $ for $ s \in [s_1, 1).$ Thus, $f(s)$ is non-increasing on $(0, s_1]$ and non-decreasing on $[s_1, 1)$. Moreover, since
$
3\alpha + 8 - 2p > \frac{\alpha + 2}{p} > 0,
$
using L'Hospital rule we have
\begin{align*}
\lim_{s \to 0^+} f(s) &= \lim_{s \to 0^+} \frac{\int_0^s \psi_{b,p,\alpha}(t) \dd t}{s^{3\alpha + 8 - 2p}} - B\Big(\frac{\alpha + 2}{p},~ b+1 - \frac{\alpha + 2}{p}\Big) \\[0.1cm]
&= \lim_{s \to 0^+} \frac{\psi_{b,p,\alpha}(s)}{(3\alpha + 8 - 2p)s^{3\alpha + 7 - 2p}} - B\Big(\frac{\alpha + 2}{p},~ b+1 - \frac{\alpha + 2}{p}\Big) \\[0.1cm]
&= \lim_{s \to 0^+} \frac{s^{2p - 3\alpha + \frac{\alpha + 2}{p} - 8}(1 - s)^{b-\frac{\alpha + 2}{p}}}{(3\alpha + 8 - 2p)} - B\Big(\frac{\alpha + 2}{p},~ b+1 - \frac{\alpha + 2}{p}\Big) \\[0.1cm]
&= +\infty.
\end{align*}
Observing that $f(1) = 0$, we deduce the existence of a value $c \in (0, s_1)$ such that $f(s)>0$ on $(0,c)$ and $f(s)\leq0$ on $ [c, 1]$.

 {\bf Case (iii)}. Since $b>\frac{\alpha + 2}{p}$ and $2p - 3\alpha - 8 \geq 0$, for $s\in[0,1]$, we have $h(0)>0,$  $h(1)<0.$ Thus, there exists a $s_0 \in (0,1)$ such that
\begin{align*}
h(s) &\geq 0 \quad \text{for} \quad s \in [0, s_0] \quad \text{and} \quad h(s) \leq 0 \quad \text{for} \quad s \in [s_0, 1].
\end{align*}
Therefore,
\begin{align*}
g'(s) &\geq 0 \quad \text{for} \quad s \in (0, s_0] \quad \text{and} \quad g'(s) \leq 0 \quad \text{for} \quad s \in [s_0, 1).
\end{align*}
Thus, $g(s)$ is non-decreasing on $(0, s_0]$ and non-increasing on $[s_0, 1)$. Since $b>\frac{\alpha + 2}{p}$ and $2p - 3\alpha - 8 \geq 0$, it can be readily verified that
$$
\lim\limits_{s \rightarrow 1^{-}} g(s) =(2p - 3\alpha - 8)B\Big(\frac{\alpha + 2}{p},~b+1-\frac{\alpha + 2}{p}\Big)\geq0.
$$
%In fact,
% \begin{align*}
%  \lim_{s \to 1^-}g(s)&=  \lim_{s \to 1^-}(2p - 3\alpha - 8) \int_0^s \psi_{b,p,\alpha}(t) \dd t +   \lim_{s \to 1^-}s \psi_{b,p,\alpha}(s)\\[0.1cm]
%  &=(2p - 3\alpha - 8) \int_0^1 \psi_{b,p,\alpha}(t) \dd t +   \lim_{s \to 1^-}s^{\frac{\alpha+2}{p}}(1-s)^{b-\frac{\alpha+2}{p}}\\[0.1cm]
%  &=(2p - 3\alpha - 8) \int_0^1 \psi_{b,p,\alpha}(t) \dd t.
%\end{align*}
Combining with $\lim\limits_{s \rightarrow 0^{+}} g(s) =0$, we conclude that $g(s)$ and $f'(s)$ are nonnegative on $(0,1)$. Therefore, $f(s)$ is increasing on $(0,1]$, so that $f(s) \leq f(1) = 0$ for all $s$ in $(0,1]$.
\end{proof}

\begin{remark}If $p<p_{\alpha}$ and $b\ge \frac{\alpha+2}{p}$, then $f(s)\ge 0$ for all $s\in (0,1]$. Indeed, under these conditions one shows that $f'(s)\le 0$ on $(0,1]$, so $f(s)\ge f(1)$. Together with $f(1)=0$, the conclusion follows.
 \end{remark}

Let $\alpha=1$. We have the following lemma.

\begin{lemma}\label{lem5.1}  Let   $b\geq 0$, $1< p< \infty$  such that  $p > \frac{3}{b+1}$. Set
\begin{align}\label{eq5.1}
f_{b,p}(s) &= s^{2p-11} \int_0^s \psi_{b,p,1}(t) \dd t - s^{2p-12}(1-s) \int_0^s \frac{t \psi_{b,p,1}(t)}{1-t} \dd t - \int_0^1 \psi_{b,p,1}(t) \dd t.
\end{align}
\begin{enumerate}
\item [{\bf(i)}] If $p \geq \frac{1}{4}(11 + \sqrt{97})$ and $b\leq\frac{3}{p}$, then $f_{b,p}(s) \leq 0$ for all $s$ in $(0,1)$;
    \item [{\bf(ii)}] If $p \geq \frac{1}{4}(11 + \sqrt{97}), 2p-11\geq0$ and $b>\frac{3}{p}$, then $f_{b,p}(s) \leq 0$ for all $s$ in $(0,1)$;
\item [{\bf(iii)}] If $p < \frac{1}{4}(11 + \sqrt{97})$ and $b<\frac{3}{p}$,  then there exists some constant $c \in (0,1)$ such that $f_{b,p}(s) > 0$ for $s \in (0,c)$ and $f_{b,p}(s) \leq 0$ for $s \in [c,1)$.
\end{enumerate}
\end{lemma}

\begin{proof}
  When $p \geq \frac{1}{4}(11 + \sqrt{97})$, the desired conclusion follows immediately from Lemma \ref{lem308-6} for $\alpha = 1$.

When $p < \frac{1}{4}(11 + \sqrt{97})$, namely $2p + \frac{3}{p} < 11$, we rewrite
\begin{align}\label{eq5.2}
f_{b,p}(s)  %&= s^{2p-11} \int_0^s \psi_{b,p,1}(t) \dd t - s^{2p-12}(1-s) \int_0^s \frac{t \psi_{b,p,1}(t)}{1-t} \dd t - \int_0^1 \psi_{b,p,1}(t) \dd t\nonumber \\[0.1cm]
&= s^{2p-11} \int_0^s \psi_{b,p,1}(t) \dd t - s^{2p-12}(1-s) \left( \int_0^s \frac{\psi_{b,p,1}(t)}{1-t} \dd t - \int_0^s \psi_{b,p,1}(t) \dd t \right) - \int_0^1 \psi_{b,p,1}(t) \dd t\nonumber \\[0.1cm]
&= s^{2p-12} \int_0^s \psi_{b,p,1}(t)\dd t - s^{2p-12}(1-s) \int_0^s \frac{\psi_{b,p,1}(t)}{1-t} \dd t - \int_0^1 \psi_{b,p,1}(t) \dd t.
\end{align}
Taking the derivative of both sides simultaneously yields
\begin{align*}
f'_{b,p}(s) &= (2p-12)s^{2p-13} \int_0^s \psi_{b,p,1}(t) \dd t + s^{2p-12} \psi_{b,p,1}(s) - s^{2p-12}(1-s) \frac{\psi_{b,p,1}(s)}{1-s} \\[0.1cm]
&\quad - [(2p-12)s^{2p-13} - (2p-11)s^{2p-12}] \int_0^s \frac{\psi_{b,p,1}(t)}{1-t} \dd t\\[0.1cm]
& = s^{2p-13}g_{b,p}(s),
\end{align*}
where
\begin{align*}
g_{b,p}(s) &= (2p - 12) \int_0^s \psi_{b,p,1}(t) \dd t - \big[(2p - 12) - (2p - 11)s\big] \int_0^s \frac{\psi_{b,p,1}(t)}{1 - t} \dd t.
\end{align*}
Differentiating the above equation and recall that $\psi_{b,p,1}(s)=s^{\frac{3}{p}-1}(1 - s)^{b-\frac{3}{p}}$ we get
\begin{align*}
g'_{b,p}(s) &= (2p - 12)\psi_{b,p,1}(s) - [(2p - 12) - (2p - 11)s] \frac{\psi_{b,p,1}(s)}{1 - s} + (2p - 11) \int_0^s \frac{\psi_{b,p,1}(t)}{1 - t} \dd t \\[0.1cm]
&= \frac{s \psi_{b,p,1}(s)}{1 - s} + (2p - 11) \int_0^s \frac{\psi_{b,p,1}(t)}{1 - t} \dd t \\[0.1cm]
& = s^{\frac{3}{p}}(1 - s)^{b-\frac{3}{p} - 1} + (2p - 11) \int_0^s \frac{\psi_{b,p,1}(t)}{1 - t} \dd t.
\end{align*}
To determine the sign of \( g'_{b,p} \), we compute the second derivative of \( g_{b,p}\), i.e.,
\begin{align*}
g_{b,p}''(s) &= \frac{3}{p} s^{\frac{3}{p} - 1}(1 - s)^{b-\frac{3}{p} - 1} + \Big(\frac{3}{p} + 1-b\Big) s^{\frac{3}{p}}(1 - s)^{b-\frac{3}{p} - 2} + (2p - 11) \frac{\psi_{b,p,1}(s)}{1 - s} \\[0.1cm]
&= \frac{3}{p} s^{\frac{3}{p} - 1}(1 - s)^{b-\frac{3}{p} - 1} + \Big(\frac{3}{p} + 1-b\Big) s^{\frac{3}{p}}(1 - s)^{b-\frac{3}{p} - 2} + (2p - 11) s^{\frac{3}{p} - 1}(1-s)^{b-\frac{3}{p} - 1}\\[0.1cm]
%& = s^{\frac{3}{p} - 1}(1 - s)^{b-\frac{3}{p} - 2} \Big[\frac{3}{p}(1-s)+\Big(\frac{3}{p}+1-b\Big)s+(2p-11)(1-s)\Big]\\[0.1cm]
&= s^{\frac{3}{p} - 1}(1 - s)^{b-\frac{3}{p} - 2} \Big[\left(12-2p-b\right)s+2p+\frac{3}{p}-11\Big]\\[0.1cm]
&:= s^{\frac{3}{p} - 1}(1 - s)^{b-\frac{3}{p} - 2} h_{b,p}(s).
\end{align*}
Since $2p + \frac{3}{p} < 11$ and  $b<1+\frac{3}{p}$, we have $$ h_{b,p}(0)=\frac{3}{p}+2p-11<0 ~~~\mbox{and}~~~ h_{b,p}(1)=\frac{3}{p}+1-b>0.$$
 This means that there is $s_0 \in (0,1)$ such that
 $$h_{b,p}(s) \leq 0~~~ \mbox{on} ~~~[0, s_0] ~~~\mbox{and}~~~h_{b,p}(s) \geq 0 ~~~\mbox{ on}~~~[s_0, 1].$$
  Thus
   $$g''_{b,p}(s) \leq 0 ~~~\mbox{ on}~~~(0, s_0]~~ ~\mbox{and}~~~ g''_{b,p}(s) \geq 0~~~\mbox{ on}~~~[s_0, 1),$$
   which shows that $g'_{b,p}(s)$ is decreasing on $(0, s_0]$ and increasing on $[s_0, 1)$. Clearly, $g'_{b,p}(0) = 0$.

 On the other hand, since $b<\frac{3}{p}$ and
\begin{align}\label{eq5.3}
\lim_{s \to 1^-} \frac{\int_0^s \frac{\psi_{b,p,1}(t)}{1-t} \dd t}{(1-s)^{b-\frac{3}{p}}} & = \lim_{s \to 1^-} \frac{\frac{\psi_{b,p,1}(s)}{1-s}}{\left(\frac{3}{p}-b\right)(1-s)^{b-\frac{3}{p} - 1}} = \frac{p}{3-bp},
\end{align}
we conclude that $g'_{b,p}(s) \to +\infty$ as $s \to 1$. In fact,
 \begin{align*}
  \lim_{s \to 1^-} g'_{b,p}(s)&=\lim_{s \to 1^-} \left(s^{\frac{3}{p}}(1 - s)^{b-\frac{3}{p} - 1} + (2p - 11) \int_0^s \frac{\psi_{b,p,1}(t)}{1 - t} \dd t\right) =+\infty.
\end{align*}
Hence there exists $s_1 \in (s_0, 1)$ such that $g'_{b,p}(s) \leq 0$ on $(0, s_1]$ and $g'_{b,p}(s) \geq 0$ on $[s_1, 1)$. It follows that $g_{b,p}(s)$ is decreasing on $(0, s_1]$ and increasing on $[s_1, 1)$. Observing that $g_{b,p}(0) = 0$ and
% $g_{b,p}(s) \to +\infty$ as $s \to 1$,
 \begin{align*}
  \lim_{s \to 1^-} g_{b,p}(s) &=\lim_{s \to 1^-} \left((2p - 12) \int_0^s \psi_{b,p,1}(t) \dd t - [(2p - 12) - (2p - 11)s] \int_0^s \frac{\psi_{b,p,1}(t)}{1 - t} \dd t\right)\\[0.1cm]
  &=+\infty,
\end{align*}
 we deduce that there exists $s_2 \in (s_1, 1)$ such that $f'_{b,p}(s) \leq 0 ~~\mbox{ on}~~ [0, s_2]$
   and
 $f'_{b,p}(s) \geq 0 ~~~\mbox{ on}~~[s_2, 1).$
   Thus $f_{b,p}(s)$ is decreasing on $(0, s_2]$ and increasing on $[s_2, 1)$. By (\ref{eq5.2}) and (\ref{eq5.3}) we have
\begin{align*}
\lim\limits_{s\rightarrow 1^-} f_{b,p}(s)&=\lim\limits_{s\rightarrow 1^-}s^{2p-12} \int_0^s \psi_{b,p,1}(t)\dd t -\lim\limits_{s\rightarrow 1^-} s^{2p-12}(1-s) \int_0^s \frac{\psi_{b,p,1}(t)}{1-t} \dd t - \int_0^1 \psi_{b,p,1}(t) \dd t\\[0.1cm]
&= -\lim\limits_{s\rightarrow 1^-} s^{2p-12}(1-s)^{b+1-\frac{3}{p}} \frac{\int_0^s \frac{\psi_{b,p,1}(t)}{1-t}\dd t}{(1-s)^{b-\frac{3}{p}}}   =0.
\end{align*}
Let
\begin{align*}
l_{b,p}(s) &= \frac{\int_0^s \psi_{b,p,1}(t) \dd t}{s^{\frac{3}{p}}} - \frac{(1-s) \int_0^s \frac{t \psi_{b,p,1}(t)}{1-t} \dd t}{s^{\frac{3}{p} + 1}} - s^{11 - 2p - \frac{3}{p}} \int_0^1 \psi_{b,p,1}(t) \dd t.
\end{align*}
Noting that $(11 - 2p - \frac{3}{p}) > 0$, we have
\begin{align*}
\lim_{s \to 0^+} l_{b,p}(s) &= \lim_{s \to 0^+} \frac{\psi_{b,p,1}(s)}{\frac{3}{p} s^{\frac{3}{p} - 1}} - \lim_{s \to 0^+} \frac{s \psi_{b,p,1}(s) - \int_0^s \frac{t \psi_{b,p,1}(t)}{1-t} \dd t}{\left(\frac{3}{p} + 1\right) s^{\frac{3}{p}}}\nonumber\\[0.1cm]
& = \frac{p}{3} - \frac{p}{3 + p} = \frac{p^2}{3(3 + p)} > 0.
\end{align*}
 Therefore, by (\ref{eq5.1}) we get that $$f_{b,p}(s) = s^{2p + \frac{3}{p} - 11} l_{b,p}(s) \to +\infty$$ as $s \to 0^+$. Combining these, we conclude that there exists $c \in (0, s_2)$ such that $f_{b,p}(s) > 0$ on $(0, c)$ and $f_{b,p}(s) \leq 0$ on $[c, 1)$. The proof is complete.
\end{proof}

When \( \alpha = 0 \), the following result analogous to that of Lemma \ref{lem5.1} holds true.   Since its proof proceeds along the same lines of reasoning as Lemma \ref{lem5.1}, we shall omit it here.

\begin{lemma}  \label{l28} Let   $b\geq 0$, $1< p< \infty$  such that $p > \frac{2}{b+1}$. Set
\begin{align*}
f_{b,p}(s) &= s^{2p-11} \int_0^s \psi_{b,p,0}(t) \dd t - s^{2p-12}(1-s) \int_0^s \frac{t \psi_{b,p,0}(t)}{1-t} \dd t - \int_0^1 \psi_{b,p,0}(t) \dd t.
\end{align*}
\begin{enumerate}
\item [{\bf(i)}] If $p \geq2 + \sqrt{3}$ and $b\leq\frac{2}{p}$, then $f_{b,p}(s) \leq 0$ for all $s \in (0,1)$;
\item [{\bf(ii)}] If $p \geq2 + \sqrt{3}, 2p-8\geq0$ and $b>\frac{2}{p}$, then $f_{b,p}(s) \leq 0$ for all $s \in (0,1)$;
\item [{\bf(iii)}] If $p < 2 + \sqrt{3}$ and $b<\frac{2}{p}$, then there exists some constant $c \in (0,1)$ such that $f_{b,p}(s) > 0$ for $s \in (0,c)$ and $f_{b,p}(s) \leq 0$ for $s \in [c,1)$.
\end{enumerate}
\end{lemma}

The following inequality may be familiar to some experts, yet no  literature has been identified. For the sake of completeness, we present a detailed proof herein.

\begin{lemma} \label{apinq} Let $1< p< \infty$ and $\alpha>-1$. Then for every \( n \in \mathbb{N}_0 \),
\begr %\label{fneq}
    |f^{(n)}(0) |^p\leq \frac{(n!)^p   \| f   \|_{\apa}^p}{(\alpha+1)B\Big(\frac{np}{2}+1,\alpha+1 \Big)}. \nonumber %\label{eq:main-estimate}
   \endr
 \end{lemma}

\begin{proof} Let $f(z) = \sum_{n=0}^\infty a_n z^n \in A_\alpha^p .$  By Cauchy's integral formula,
\[
a_n = \frac{1}{2\pi i} \int_{|z|=r} \frac{f(z)}{z^{n+1}} \, \dd z.
\]
Substituting \( z = re^{i\theta} \), we have \( \dd z = i r e^{i\theta} \dd \theta \). Hence
\[
a_n = \frac{1}{2\pi r^n} \int_{-\pi}^{\pi} f(re^{i\theta}) e^{-in\theta} \, \dd \theta .
\]
Taking absolute values and applying H\"older's inequality, we obtain
\begin{align}
r^n |a_n| \le \frac{1}{2\pi} \int_{-\pi}^{\pi} |f(re^{i\theta})| \, \dd \theta \le (2\pi)^{-\frac{1}{p}}
\left( \int_{-\pi}^{\pi} |f(re^{i\theta})|^p \, \dd \theta \right)^{1/p}. \label{eq:cauchy-estimate}
\end{align}
Recall that \( a_n = f^{(n)}(0)/n! \),  the   inequality (\ref{eq:cauchy-estimate}) is equivalent to
\begin{align}
r^{np} \frac{|f^{(n)}(0)|^p}{(n!)^p}
\le \frac{1}{2\pi} \int_{-\pi}^{\pi} |f(re^{i\theta})|^p \, \dd\theta . \label{eq:pointwise-estimate11}
\end{align}
Multiply both sides of \eqref{eq:pointwise-estimate11} by $(\alpha+1)r(1-r^2)^\alpha$, then integrate the resulting inequality with respect to $r$ over the interval $[0,1]$. We get
\begin{align}
\frac{|f^{(n)}(0)|^p}{(n!)^p} (\alpha+1) \int_0^1 r^{np+1} (1-r^2)^\alpha \,\dd r
&\le \frac{\alpha+1}{2\pi}
\int_0^1 \int_{-\pi}^{\pi} |f(re^{i\theta})|^p \, \dd \theta \,
r (1-r^2)^\alpha \, \dd r \notag \\
&= \frac{1}{2} \|f\|_{A_\alpha^p}^p . \label{eq:integrated-estimate}
\end{align}
We next simplify the radial integral appearing on the left-hand side via the change of variable $t = r^2$, so that $r = \sqrt{t}$ and $\dd r = \frac{\dd t}{2\sqrt{t}}$. Direct substitution yields
\begin{align}
(\alpha+1) \int_0^1 r^{np+1} (1-r^2)^\alpha \, \dd r
%&= (\alpha+1) \int_0^1 t^{\frac{np+1}{2}} (1-t)^\alpha \frac{\dd t}{2\sqrt{t}} \notag \\
= \frac{\alpha+1}{2} \int_0^1 t^{\frac{np}{2}} (1-t)^\alpha \, \dd t  = \frac{\alpha+1}{2} B\Big( \frac{np}{2}+1, \alpha+1 \Big). \label{eq:beta-integral}
\end{align}
By \eqref{eq:beta-integral} and \eqref{eq:integrated-estimate}, we get the desired result,  completing the proof of this lemma.
\end{proof}

We shall need the norm of the composition operator \(C_\varphi\), defined by
$C_\varphi f = f \circ \varphi,~f\in \mathrm{H}(\mathbb{D}) $,  where \(\varphi\) is an analytic self-map of the
unit disk. For affine symbols \(\varphi(z)=\rho z+c\) with \(|\rho|+|c|\le 1\),
its norm on the weighted Bergman space \(A_\alpha^p\) is given by the following
lemma (see Theorem 6.4 in \cite{Da}).

\begin{lemma}[\cite{Da}]\label{lem66}
  Let $\varphi(z)=\rho z+c(\rho \ne 0)$, where $\rho$ and $c$ are complex numbers satisfying $\left |\rho \right | +\left | c \right | \le 1$. Then the norm of the composition operator $C_{\varphi}$ on the weighted Bergman space $A_{\alpha}^{p}$ with $p>0$ and $\alpha>-1$ is equal to
  \begin{align*}
  \| C_{\varphi} \|_{\apa\rightarrow\apa}=\Big( \frac{2}{1+|\rho|^{2}-|c|^{2} +\sqrt{ (1+\ |\rho|^{2}- |c|^{2})^2-4|\rho|^2}  } \Big )^{\frac{\alpha+2}{p}}.
  \end{align*}
\end{lemma}
\vskip 5mm

  \section{$\|\mathcal{H}_b\|_{\apa\rightarrow\apa }$ for $\alpha\geq0$ }

In this section, we give the upper and lower bounds for the norm of the generalized Hilbert operator $\mathcal{H}_b$ on weighted Bergman spaces  $\apa$ when $\alpha\geq0$. Under certain conditions, the exact value of the operator norm is obtained.

For any real number $x$, recall that the ceiling function of $x$ is formally defined as
\[
\lceil x \rceil := \min\left\{n \in \mathbb{Z} \mid n \geq x\right\}.
\]
Throughout the remainder of this paper, we denote the backward shift operator on weighted Bergman spaces via
\[
S^{*}(f)(z) =
\begin{cases}
\dfrac{f(z) - f(0)}{z}, & z \neq 0 \\
f'(0), & z = 0.
\end{cases}
\]

\begin{theorem}\label{thm3}
  Let $b, \alpha\geq0$, $1< p< \infty$ such that $ \alpha+2<p(b+1)   $.
  \begin{enumerate}
		\item[{\bf(i)}] If $p\geq2(\alpha+2)$, then
$$
\|\mathcal{H}_b\|_{\apa\rightarrow\apa}\leq B\Big(\frac{\alpha+2}{p},~b+1-\frac{\alpha+2}{p}\Big).
$$
		\item[{\bf(ii)}] If $\alpha +3\leq p<2(\alpha +2)$, then
$$
\|\mathcal{H}_b\|_{\apa\rightarrow\apa}\leq2^{\frac{\alpha+1}{p}}B\Big(\frac{\alpha+2}{p},~b+1-\frac{\alpha+2}{p}\Big).
$$
		\item[{\bf(iii)}] If $\alpha+2<p<\alpha +3$, then
$$
\|\mathcal{H}_b\|_{\apa\rightarrow\apa}\leq B\Big(\frac{\alpha+2}{p},~b+1-\frac{\alpha+2}{p}\Big)+2^{\frac{2(\alpha+2)}{p}-1}B\Big(1-\frac{\alpha+2}{p},~b+\frac{\alpha+2}{p}\Big).
$$
        \item[{\bf(iv)}] If $\alpha>0, b>0$, $\max\left \{ \frac{\alpha+2}{b+1},2\right \}<p\le \alpha+2   $, then
$$
\|\mathcal{H}_b\|_{\apa\rightarrow\apa}\leq 2^{\frac{2(\alpha+2)}{p}-1}B\Big(\frac{\alpha+2}{p},~b+1-\frac{\alpha+2}{p}\Big)+2^{\frac{3\alpha+8}{p}-2}B\Big(1-\frac{2}{p},~b+\frac{2}{p}\Big).
$$

 \item[{\bf(v)}]  If $ b>0\ and\ \frac{\alpha+2}{b+1}<p\le 2   $, then
{\small \begin{align*}
& \|\mathcal{H}_b\|_{\apa\rightarrow\apa}\\[0.1cm]
 \leq &\sum_{n=0}^{{\lceil b\rceil}-1}\left[\left(\alpha+1 \right)B\Big(\frac{np}{2}+1,\alpha+1 \Big) \right]^{-\frac{1}{p}}\int_{0}^{1}(1-s)^{b}s^{n}\left[F\left(\frac{p(n+1)}{2},\frac{p(n+1)}{2},\alpha+2;(1-s)^{2} \right) \right]^{\frac{1}{p}}\dd s\nonumber \\[0.1cm]
&+B\left(\frac{\alpha+2}{p},b+1-\frac{\alpha+2}{p} \right)\left[\frac{(\alpha+1)2^{2\alpha+7-({\lceil b\rceil}+1)p}}{9[({\lceil b\rceil}+1)p-2-\alpha] }+2^{\alpha}\max\left \{ 1,2^{\alpha+4-({\lceil b\rceil}+1)p} \right \}  \right]^{\frac{1}{p}}\left \| S^{*} \right \|_{\apa\rightarrow\apa}^{\lceil b\rceil}.
\end{align*} }

\end{enumerate}
\end{theorem}

\begin{proof}   Let $f\in\apa$.   The operator $\mathcal{H}_b$ has an integral representation as follows:
\begin{align*}
\mathcal{H}_b(f)(z) = \int_{0}^{1} T_s(f)(z) \,~~~~ \dd s,\,\,0<s<1,\,\,\,z\in\D,
\end{align*}
where $
   T_s(f)(z)=w_s(z)f(\phi_s(z)) $ with
    \begin{align}\label{eq-690}
  w_s(z)=\frac{(1-s)^b}{(s-1)z+1},~~~~~ \phi_s(z)=\frac{s}{(s-1)z+1}.
  \end{align}
 From the  Minkowski's inequality, it follows that
  \begin{align}\label{eq1}
    \|\mathcal{H}_b(f)\|_{\apa}&=\left((\alpha+1)\int_{\D}|\mathcal{H}_bf(z)|^p(1-|z|^2)^{\alpha}\daz\right)^{\frac{1}{p}}\nonumber\\[0.1cm]
    &=(\alpha+1)^{\frac{1}{p}}\left(\int_{\D}\left|\int_0^1T_s(f)(z)(1-|z|^2)^{\frac{\alpha}{p}}\dd s\right|^p\daz\right)^{\frac{1}{p}}\nonumber\\[0.1cm]
    &\leq(\alpha+1)^{\frac{1}{p}}\int_0^1\left(\int_{\D}\left|T_s(f)(z)\right|^p(1-|z|^2)^{\alpha}\daz\right)^{\frac{1}{p}}\dd s\nonumber\\[0.1cm]
    &=\int_0^1\|T_s(f)\|_{\apa}\dd s.
  \end{align}
It is clear that
\begin{align}\label{eq-693}
  \phi_s(\D)=D\left(\frac{1}{2-s},\frac{1-s}{2-s}\right):=D_{s}
\end{align}
is the open disc centered at $\dfrac{1}{2-s}$ with radius $\dfrac{1-s}{2-s}$, and for every $0<s<1$, $\phi_s(\D)\subset\D$ with $$\overline{\phi_s(\D)}\cap\partial\D=\{1\}.$$
  From (\ref{eq-690}),
 it is easy to check that
      \begin{align}\label{eq316}
  \phi_s^{-1}(z)=\frac{z-s}{(1-s)z}, ~~~~~~~~~    w_s(\phi_s^{-1}(z))=\frac{z}{s}(1-s)^b,
  \end{align}
and
  \begin{align}\label{eqp}
  |\phi_s'(z)|^2=\frac{s^2}{(1-s)^{4b-2}}|w_s(z)|^4.
  \end{align}
Then, substituting the variable $z=\phi_s^{-1}(w)$ and combining with  \eqref{eq316} and \eqref{eqp}, we obtain
  \begin{align}\label{eq6006}
    \|T_s(f)\|^p_{\apa}&=(\alpha+1)\int_{\D}|w_s(z)|^{p}|f(\phi_s(z))|^p\za\daz\nonumber\\[0.1cm]
   % &=\frac{\alpha+1}{s^2(1-s)^{2-4b}}\int_{\D}|w_s(z)|^{p-4}\frac{s^2}{(1-s)^{4b-2}}|w_s(z)|^{4}|f(\phi_s(z))|^p\za\daz\nonumber\\[0.1cm]
    &=\frac{\alpha+1}{s^2(1-s)^{2-4b}}\int_{\D}|w_s(z)|^{p-4}|\phi_s'(z)|^2|f(\phi_s(z))|^p\za\daz\nonumber\\[0.1cm]
    &=\frac{\alpha+1}{s^2(1-s)^{2-4b}}\int_{\phi_s(\D)}\frac{|w|^{p-4}}{s^{p-4}}(1-s)^{pb-4b}|f(w)|^p(1-|\phi_s^{-1}(w)|^2)^\alpha\dd A(w)\nonumber\\[0.1cm]
    &=\frac{\alpha+1}{s^{p-2}(1-s)^{2-pb}}\int_{\phi_s(\D)}|z|^{p-4}|f(z)|^p(1-|\phi_s^{-1}(z)|^2)^\alpha\daz.
  \end{align}

  {\bf Case~(i) $p\geq2(\alpha+2)$.}

   A simple computation shows that
  \begin{align}\label{eq-var}
    1-|\phi_s^{-1}(z)|^2&=\frac{(1-s)^2|z|^2-|z-s|^2}{(1-s)^2|z|^2}  =\frac{2s\operatorname{Re} z-s^2-s(2-s)|z|^2}{(1-s)^2|z|^2}\nonumber\\[0.1cm]
    &=\frac{s}{1-s}\frac{2\operatorname{Re} z-s-(2-s)|z|^2}{(1-s)|z|^2}\nonumber\\[0.1cm]
    & :=\frac{s}{1-s}g_s(z),\,\,\,z\in\phi_s({\D}).
  \end{align}
 It is clear that
  \begin{align}\label{eq66}
     g_s(z)&\leq\frac{2|z|-s-(2-s)|z|^2}{(1-s)|z|^2} \leq\frac{1+|z|^2-s-(2-s)|z|^2}{(1-s)|z|^2}=\frac{1-|z|^2}{|z|^2}.
  \end{align}
 Combining (\ref{eq6006}), (\ref{eq-var}) and (\ref{eq66}) we get
 \begin{align}\label{eq2}
    \|T_s(f)\|^p_{\apa}
    %=\frac{\alpha+1}{s^{p-2-\alpha}(1-s)^{2-pb+\alpha}}\int_{\phi_s(\D)}|z|^{p-4}|f(z)|^p\left(g_s(z)\right)^\alpha\daz\nonumber\\[0.1cm]
    \leq\frac{\alpha+1}{s^{p-2-\alpha}(1-s)^{2-pb+\alpha}}\int_{\phi_s(\D)}|z|^{p-4-2\alpha}|f(z)|^p\za\daz.
  \end{align}
  Hence,
  \begin{align*}
    \|T_s(f)\|^p_{\apa}
    &\leq\frac{\alpha+1}{s^{p-2-\alpha}(1-s)^{2-pb+\alpha}}\int_{\D}|f(z)|^p\za\daz\\[0.1cm]
    &=\frac{1}{s^{p-2-\alpha}(1-s)^{2-pb+\alpha}}\|f\|^p_{\apa},
  \end{align*}
  which implies that
  \begin{align}\label{eq303}
    \|T_s(f)\|_{\apa}&\leq s^{\frac{\alpha+2}{p}-1}(1-s)^{b-\frac{\alpha+2}{p}}\|f\|_{\apa}=\psi_{b,p,\alpha}(s)\|f\|_{\apa}.
  \end{align}
  Therefore, when $p\geq2(\alpha+2)$, by (\ref{eq303}) and (\ref{eq1}) we get
  \begin{align*}
    \|\mathcal{H}_b(f)\|_{\apa}&\leq B\Big(\frac{\alpha+2}{p},~b+1-\frac{\alpha+2}{p}\Big)\|f\|_{\apa}.
  \end{align*}

    {\bf Case~(ii) $\alpha+3\leq p<2(\alpha+2)$.}

     Let
\begin{align}\label{eq126}
    \rho_s=\frac{1-s}{2-s}\,\,\,\text{and}\,\,\,c_s=\frac{1}{2-s},\,\,\,0<s<1.
\end{align}
It follows from (\ref{eq-693}) that
 \begin{align}\label{eq-694}
  D_{s}=D\left(c_s,\rho_s\right)=\phi_s(\D)\subset R_{s^2}=\{z\in \mathbb{C}:s^2<|z|<1\}.
 \end{align}
 and
    \begin{align}\label{eq919}
      1-|\phi^{-1}_s(z)|^2=\frac{s}{1-s}\frac{1}{|z|^2}\frac{\rho_s^2-|z-c_s|^2}{\rho_s},\,\,s\in(0,1).
    \end{align}
For $z\in D_s$,  a simple calculation   shows that
 \begin{align*}
   \frac{\rho_s^2-|z-c_s|^2}{\rho_s}&=\frac{\rho_s^2-|z|^2-c_s^2+2c_s\operatorname{Re}z}{\rho_s} \leq\frac{\rho_s^2-|z|^2-c_s^2+2c_s|z|}{\rho_s}\nonumber\\[0.1cm]
   &=\frac{\rho_s-c_s-|z|^2+(1-\rho_s+c_s)|z|}{\rho_s}=(1-|z|)\frac{(c_s-\rho_s+2\rho_s)|z|+\rho_s-c_s}{\rho_s}\nonumber\\[0.1cm]
   &=(1-|z|)\left[2|z|-\frac{(c_s-\rho_s)(1-|z|)}{\rho_s}\right]   \leq2|z|(1-|z|).
 \end{align*}
Together this with $\alpha\geq0$, we obtain
 \begin{align}\label{eq305-1}
  \left( \frac{\rho_s^2-|z-c_s|^2}{\rho_s}\right)^\alpha\leq2^\alpha|z|^\alpha(1-|z|)^\alpha.
 \end{align}
Combining (\ref{eq6006}), (\ref{eq-694}),  (\ref{eq919}) and (\ref{eq305-1}) we have
\begin{align*}%\label{eq304-4}
    \|T_s(f)\|^p_{\apa}
    &=\frac{\alpha+1}{s^{p-2}(1-s)^{2-pb}}\int_{\phi_s(\D)}|z|^{p-4}|f(z)|^p(1-|\phi_s^{-1}(z)|^2)^\alpha\daz\nonumber\\[0.1cm]
    &=\frac{\alpha+1}{s^{p-2}(1-s)^{2-pb}}\int_{D_s}|z|^{p-4}|f(z)|^p
    \left(\frac{s}{1-s}\frac{1}{|z|^2}\frac{\rho_s^2-|z-c_s|^2}{\rho_s}\right)^\alpha\daz\nonumber\\[0.1cm]
    &=\frac{\alpha+1}{s^{p-2-\alpha}(1-s)^{2-pb+\alpha}}\int_{D_s}|z|^{p-4-2\alpha}|f(z)|^p
    \left(\frac{\rho_s^2-|z-c_s|^2}{\rho_s}\right)^\alpha\daz\nonumber\\[0.1cm]
    &\leq\frac{2^\alpha(\alpha+1)}{s^{p-2-\alpha}(1-s)^{2-pb+\alpha}}\int_{\mathbb{D}}|z|^{p-4-\alpha}|f(z)|^p
    \left(1-|z|\right)^\alpha\daz.
  \end{align*}
  Since $\left | z \right |^{p-4-\alpha}\le \left | z \right |^{-1}$ holds for all $z\in \mathbb{D}$ with $z\neq 0$, substituting this into   (3.14) yields
  \begin{align*}
    \|T_s(f)\|^p_{\apa}
    &\leq\frac{2^\alpha(\alpha+1)}{s^{p-2-\alpha}(1-s)^{2-pb+\alpha}}\int_{\mathbb{D}}|z|^{-1}|f(z)|^p
    \left(1-|z|\right)^\alpha\daz\nonumber\\[0.1cm]
    &=\frac{2^\alpha(\alpha+1)}{s^{p-2-\alpha}(1-s)^{2-pb+\alpha}}\int_{0}^1\frac{2}{2\pi}\int_{0}^{2\pi} |f(re^{i\theta})|^p
    \dd \theta \left(1-r\right)^\alpha \dd r\nonumber\\[0.1cm]
    &=\frac{2^{\alpha+1}(\alpha+1)}{s^{p-2-\alpha}(1-s)^{2-pb+\alpha}}\int_{0}^1M_{p}^{p}(r,f)\left(1-r\right)^\alpha \dd r\nonumber\\[0.1cm]
    &=\frac{2^{\alpha+1}(\alpha+1)}{s^{p-2-\alpha}(1-s)^{2-pb+\alpha}}2\int_{0}^1M_{p}^{p}(r^{2},f)\left(1-r^{2}\right)^\alpha r\dd r\nonumber\\[0.1cm]
     &\leq\frac{2^{\alpha+1}}{s^{p-2-\alpha}(1-s)^{2-pb+\alpha}}
    (\alpha+1)2\int_{0}^{1} M_{p}^{p}(r,f)\left(1-r^2 \right)^\alpha  r\dd r \nonumber\\[0.1cm]
    &=\frac{2^{\alpha+1}}{s^{p-2-\alpha}(1-s)^{2-pb+\alpha}} \left \| f \right \|_{A_{\alpha}^{p}}^{p} .
  \end{align*}
Taking the $p$-th root of both sides of the above inequality immediately gives
\begin{align*} %\label{eq4t4}
    \|T_s(f)\|_{A_{\alpha}^{p}}\leq2^{\frac{\alpha+1}{p}}\psi_{b,p,\alpha}(s)\|f\|_{A_{\alpha}^{p}}.
\end{align*}
Therefore, whenever $\alpha+3\le p<2(\alpha+2)$, we have
\begin{align*}
    \|\mathcal{H}_b(f)\|_{\apa}&\leq2^{\frac{\alpha+1}{p}}B\Big(\frac{\alpha+2}{p},~b+1-\frac{\alpha+2}{p}\Big)\|f\|_{\apa}.
  \end{align*}

   {\bf Case~(iii) $\alpha+2<p<\alpha+3$.}

   Note that $|z|\geq\dfrac{s}{2-s}$ for $z\in\phi_s(\D)$. Thus,
   $$
   |z|^{p-2\alpha-4}\leq\Big(\frac{s}{2-s}\Big)^{p-2\alpha-4}=\Big(\frac{ 2-s}{s}\Big)^{2\alpha+4-p},\,\,\,z\in\phi_s(\D).
   $$
 So, by (\ref{eq2}) we have
   \begin{align*}
    \|T_s(f)\|^p_{\apa}
    &\leq\frac{\alpha+1}{s^{p-2-\alpha}(1-s)^{2-pb+\alpha}}\Big(\frac{s}{2-s}\Big)^{p-2\alpha-4}\int_{\D}|f(z)|^p\za\daz\\[0.1cm]
    &=\frac{(2-s)^{2\alpha+4-p}}{s^{\alpha+2}(1-s)^{2-pb+\alpha}}\|f\|^p_{\apa},
  \end{align*}
   which implies that
  \begin{align}\label{eq5}
    \|T_s(f)\|_{\apa}&\leq\frac{(2-s)^{\frac{2(\alpha+2)}{p}-1}}{s^{\frac{\alpha+2}{p}}(1-s)^{\frac{\alpha+2}{p}-b}}\|f\|_{\apa}.
  \end{align}
  Note that $\alpha+2<p<\alpha+3$ implies $0<\frac{2(\alpha+2)}{p}-1<1$.   It follows from the fact that
  \begin{align}\label{eqxy} (x+y)^\beta\leq x^\beta+y^\beta,\,\,\,x,\,y\geq0,\,\,\,\beta\in(0,1],\end{align}
   we get
  \begin{align}\label{eq7}
    \frac{(2-s)^{\frac{2(\alpha+2)}{p}-1}}{s^{\frac{\alpha+2}{p}}(1-s)^{\frac{\alpha+2}{p}-b}}&=\frac{(s+2(1-s))^{\frac{2(\alpha+2)}{p}-1}}{s^{\frac{\alpha+2}{p}}(1-s)^{\frac{\alpha+2}{p}-b}}\nonumber\\[0.1cm]
  &\leq\frac{s^{\frac{2(\alpha+2)}{p}-1}+2^{\frac{2(\alpha+2)}{p}-1}(1-s)^{\frac{2(\alpha+2)}{p}-1}}{s^{\frac{\alpha+2}{p}}(1-s)^{\frac{\alpha+2}{p}-b}}\nonumber\\[0.1cm]
  &=\psi_{b,p,\alpha}(s)+2^{\frac{2(\alpha+2)}{p}-1}s^{-\frac{\alpha+2}{p}}(1-s)^{b+\frac{\alpha+2}{p}-1}.
  \end{align}
  Therefore, by (\ref{eq1}), (\ref{eq5}) and (\ref{eq7}),
  \begin{align*}
    \|\mathcal{H}_b(f)\|_{\apa}&\leq\Big[B\Big(\frac{\alpha+2}{p},~b+1-\frac{\alpha+2}{p}\Big)+2^{\frac{2(\alpha+2)}{p}-1}
    B\Big(1-\frac{\alpha+2}{p},~b+\frac{\alpha+2}{p}\Big)\Big]\|f\|_{\apa}.
  \end{align*}

  {\bf Case~(iv)  $\alpha>0, b>0,$~ $\max\left \{ \frac{\alpha+2}{b+1},2\right \}<p\le \alpha+2   $.}

 Since $p-4-\alpha<0$ and the inequality $|z|\geq\dfrac{s}{2-s}$ holding for arbitrary $z\in\phi_s(\mathbb{D})$, combining with  (3.14), one obtains that
  \begin{align*}
    \|T_s(f)\|^p_{\apa}
    &\leq\frac{2^\alpha(\alpha+1)}{s^{p-2-\alpha}(1-s)^{2-pb+\alpha}}\int_{\mathbb{D}}\Big(\frac{2-s}{s} \Big)^{\alpha+4-p}|f(z)|^p
    \left(1-|z|\right)^\alpha\daz\nonumber\\[0.1cm]
    &\leq\frac{2^\alpha(\alpha+1)(2-s)^{\alpha+4-p}}{s^{p-2-\alpha}(1-s)^{2-pb+\alpha}}\int_{\mathbb{D}}|f(z)|^p
    (1-\left | z \right |^{2})^{\alpha}\daz\nonumber\\[0.1cm]
    &=\frac{2^{\alpha}(2-s)^{\alpha+4-p}}{s^{2}(1-s)^{2-pb+\alpha}}\left \| f \right \|_{\apa}^{p},
  \end{align*}
 which yields
 \begin{align}\label{eq-3.19}
    \left\| T_{s}(f) \right\|_{\apa} \le \frac{2^{\frac{\alpha}{p}}(2-s)^{\frac{\alpha+4}{p}-1}}{s^{\frac{2}{p}}(1-s)^{\frac{\alpha+2}{p}-b}}\left \| f \right \|_{\apa}.
\end{align}
It follows from the fact that
$  (x+y)^{\gamma}\leq2^{\gamma} \left(x^{\gamma}+y^{\gamma}\right),\,\,\,x,\,y\geq0,\,\,\,1<{\gamma}<\infty,$ we get
\begin{align}\label{eq-3.20}
\frac{2^{\frac{\alpha}{p}}(2-s)^{\frac{\alpha+4}{p}-1}}{s^{\frac{2}{p}}(1-s)^{\frac{\alpha+2}{p}-b}}
&=\frac{2^{\frac{\alpha}{p}}(s+2(1-s))^{\frac{\alpha+4}{p}-1}}{s^{\frac{2}{p}}(1-s)^{\frac{\alpha+2}{p}-b}}\nonumber\\[0.1cm]
&\leq\frac{2^{\frac{\alpha}{p}}2^{\frac{\alpha+4}{p}-1}(s^{\frac{\alpha+4}{p}-1}+2^{\frac{\alpha+4}{p}-1}(1-s)^{\frac{\alpha+4}{p}-1})}{s^{\frac{2}{p}}(1-s)^{\frac{\alpha+2}{p}-b}}\nonumber\\[0.1cm]
&=2^{\frac{2\alpha+4}{p}-1}\psi_{b,p,\alpha}(s)+2^{\frac{3\alpha+8}{p}-2}s^{-\frac{2}{p}}(1-s)^{\frac{2}{p}+b-1}.
\end{align}
Therefore, combining \eqref{eq1}, \eqref{eq-3.19} and \eqref{eq-3.20}, we arrive at
\begin{align*}
    \|\mathcal{H}_b(f)\|_{\apa}&\leq2^{\frac{2\alpha+4}{p}-1} B\Big(\frac{\alpha+2}{p},~b+1-\frac{\alpha+2}{p}\Big)+2^{\frac{3\alpha+8}{p}-2}B\Big(1-\frac{2}{p},~b+\frac{2}{p}\Big)\|f\|_{\apa}.
  \end{align*}

  {\bf Case~(v) $b>0, \frac{\alpha+2}{b+1}<p\leq 2 $.}

   For a fixed real number $b$, denote $m=\lceil b \rceil$ to simplify notation. For any $z\in \mathbb{D}$, consider
\begin{equation*}
f(z)=\sum_{n=0}^{m-1}\frac{f^{(n)}(0)}{n!}z^{n}+z^{m}S^{*m}(f)(z),
\end{equation*}
where $S^{*m}$ stands for the $m$-fold iterate of the operator $S^{*}$.   Therefore,
\begin{align} \label{eqts}
\left \| T_s(f) \right \|_{\apa}\leq \sum_{n=0}^{m-1}\frac{1}{n!}\left \| T_s(z^{n}f^{(n)}(0)) \right \|_{\apa}+\left \| T_s(z^{m} S^{*m}(f)) \right \|_{\apa}.
\end{align}
We now proceed to treat the two items individually. By Lemmas \ref{l223} and \ref{apinq} implies that
\begin{align*}
 &  \| T_s(z^{n}f^{(n)}(0))   \|_{\apa}^p
  = (\alpha+1) |f^{(n)}(0)   |^p\int_{\mathbb{D}}\left |w_s(z) \right|^p\left |\phi_s(z) \right|^{np}(1-\left |z \right|^2)^{\alpha}\dd A(z)\nonumber \\[0.1cm]
= &(\alpha+1) |f^{(n)}(0)    |^p(1-s)^{bp}s^{np}\int_{\mathbb{D}}\frac{(1-\left |z \right|^2)^{\alpha}}{\left | 1-(1-s)z \right |^{p(n+1)}}\dd A(z)\nonumber \\[0.1cm]
= &  |f^{(n)}(0)  |^p(1-s)^{bp}s^{np}F\left( \frac{p(n+1)}{2},\frac{p(n+1)}{2},\alpha+2;(1-s)^2 \right)\nonumber \\[0.1cm]
\leq  & \frac{(n!)^p \left \| f \right \|_{\apa}^p}{(\alpha+1)B\left(\frac{np}{2}+1,~\alpha+1 \right)}(1-s)^{bp}s^{np}F\left( \frac{p(n+1)}{2},\frac{p(n+1)}{2},\alpha+2;(1-s)^2 \right).
\end{align*}
Consequently,
\begin{align}
& \sum_{n=0}^{m-1}\frac{1}{n!}   \| T_s(z^{n}f^{(n)}(0))   \|_{\apa}\nonumber \\[0.1cm]
 \leq &\sum_{n=0}^{m-1}\frac{\left[F\left( \frac{p(n+1)}{2},\frac{p(n+1)}{2},\alpha+2;(1-s)^2 \right) \right]^{\frac{1}{p}}}{\left[(\alpha+1)B\left(\frac{np}{2}+1,~\alpha+1 \right)\right]^{\frac{1}{p}}}(1-s)^{b}s^{n}  \| f   \|_{\apa} . \label{eqtso}
  \end{align}

For $\left \| T_s(z^{m} S^{*m}(f)) \right \|_{\apa}$,  making the change of variable $z=\phi_s^{-1}(w)$ and using (\ref{eq305-1}), we get
\begin{align}
 & \| T_s(z^{m} S^{*m}(f))   \|_{\apa}^p\nonumber\\[0.1cm]
= &(\alpha+1)\int_{\mathbb{D}}\left | w_s(z) \right |^p \left |\phi_s(z) \right |^{mp}\left |S^{*m}(f)(\phi_s(z)) \right |^p( 1-\left |z \right |^2)^\alpha \dd A(z)\nonumber \\[0.1cm]
= &\frac{\alpha+1}{s^{p-2}(1-s)^{2-bp}}\int_{\phi_s(\mathbb{D})}\left | w\right |^{(m+1)p-4}\left |S^{*m}(f)(w) \right |^p\left(1-\left| \phi_{s}^{-1}(w) \right|^{2} \right)^{\alpha}\dd A(w) \label{eqtsm} \\[0.1cm]
\leq & \frac{\alpha+1}{s^{p-2}(1-s)^{2-bp}}\int_{\phi_s(\mathbb{D})}\left | w\right |^{(m+1)p-4}\left |S^{*m}(f)(w) \right |^p\Big( \frac{s}{1-s}\frac{2\left|w \right|}{\left|w \right|^2}(1-\left|w \right|)\Big)^{\alpha}\dd A(w)\nonumber \\[0.1cm]
\leq &\frac{2^{\alpha}(\alpha+1)}{s^{p-\alpha-2}(1-s)^{\alpha+2-bp}}\int_{\mathbb{D}}\left | w\right |^{(m+1)p-\alpha-4}\left |S^{*m}(f)(w) \right |^p(1-\left|w \right|^2)^{\alpha}\dd A(w)\nonumber \\[0.1cm]
= &\frac{2^{\alpha}(\alpha+1)}{s^{p-\alpha-2}(1-s)^{\alpha+2-bp}}\int_{\left|w \right|\leq \frac{1}{2}}\left | w\right |^{(m+1)p-\alpha-4}\left |S^{*m}(f)(w) \right |^p(1-\left|w \right|^2)^{\alpha}\dd A(w)\nonumber \\[0.1cm]
&+\frac{2^{\alpha}(\alpha+1)}{s^{p-\alpha-2}(1-s)^{\alpha+2-bp}}\int_{\frac{1}{2}< \left|w \right|<1}\left | w\right |^{(m+1)p-\alpha-4}\left |S^{*m}(f)(w) \right |^p(1-\left|w \right|^2)^{\alpha}\dd A(w)\nonumber \\[0.1cm]
:=  &\frac{2^{\alpha}(\alpha+1)}{s^{p-\alpha-2}(1-s)^{\alpha+2-bp}}(I_1+I_2).\label{eqi1i2}
\end{align}
 It is well known that  the operator $S^{*}$ is bounded on $\apa$. Since $p>\frac{2+\alpha}{1+b}\geq \frac{2+\alpha}{1+m}$, we have
\begin{align}
I_1&=\int_{\left|w \right|\leq \frac{1}{2}}\left | w\right |^{(m+1)p-\alpha-4}  |S^{*m}(f)(w)   |^p(1-\left|w \right|^2)^{\alpha}\dd A(w)\nonumber \\[0.1cm]
&\leq \int_{\left|w \right|\leq \frac{1}{2}}\left | w\right |^{(m+1)p-\alpha-4}\frac{1}{ (1-|w|^2)^2}\dd A(w)  \| S^{*m}(f)   \|_{\apa}^p\nonumber \\[0.1cm]
&\le \frac{16}{9}\int_{\left|w \right|\leq \frac{1}{2}}\left | w\right |^{(m+1)p-\alpha-4}\dd A(w)  \| S^{*m}(f)   \|_{\apa}^p\nonumber \\[0.1cm]
&\le\frac{2^{\alpha+7-(m+1)p}}{9[(m+1)p-\alpha-2]}  \| S^{*}   \|_{\apa\rightarrow\apa}^{mp}  \| f   \|_{\apa}^p.\label{eqi1}
\end{align}
Noting that $(m+1)p-\alpha-4$ is not necessarily of a fixed sign (it may be positive or negative), we thus have
\begin{align}
I_2&=\int_{\frac{1}{2}< \left|w \right|<1}\left | w\right |^{(m+1)p-\alpha-4}\left |S^{*m}(f)(w) \right |^p(1-\left|w \right|^2)^{\alpha}\dd A(w)\nonumber \\[0.1cm]
&\le \max\left \{ 1,2^{\alpha+4-(m+1)p} \right \}\int_{\mathbb{D}}\left |S^{*m}(f)(w) \right |^p(1-\left|w \right|^2)^{\alpha}\dd A(w)\nonumber \\[0.1cm]
%&=\frac{\max\left \{ 1,2^{\alpha+4-(m+1)p} \right \}}{\alpha+1}\left \| S^{*m}(f) \right \|_{\apa}^p\nonumber \\[0.1cm]
&\leq \frac{\max\left \{ 1,2^{\alpha+4-(m+1)p} \right \}}{\alpha+1}\left \| S^{*} \right \|_{\apa\rightarrow\apa}^{mp}\left \| f \right \|_{\apa}^p. \label{eqi2}
\end{align}
Substituting   (\ref{eqi1}) and  (\ref{eqi2}) into   (\ref{eqi1i2}) gives
\begin{align*}
 &  \| T_s(z^{m} S^{*m}(f))  \|_{\apa}
  \leq  2^{\frac{{\alpha}}{p}} (\alpha+1)^{\frac{1}{p}} s^{\frac{\alpha+2}{p}-1}(1-s)^{b-\frac{\alpha+2}{p}}  (I_1+I_2)^{\frac{1}{p}}\nonumber \\[0.1cm]
\leq &s^{\frac{\alpha+2}{p}-1}(1-s)^{b-\frac{\alpha+2}{p}} \left \{ \frac{(\alpha+1)2^{2\alpha+7-(m+1)p}}{9[(m+1)p-\alpha-2]}+2^{\alpha}\max\left \{ 1,2^{\alpha+4-(m+1)p} \right \}  \right \}^{\frac{1}{p}} \| S^{*}   \|_{\apa\rightarrow\apa}^{m}  \| f   \|_{\apa}.
\end{align*}
From the above discussion, (\ref{eqts}) reduces to
\begin{align*}
&  \| T_s(f)  \|_{\apa} % \leq  \sum_{n=0}^{m-1}\frac{1}{n!}  \| T_s(z^{n}f^{(n)}(0))  \|_{\apa}+  \| T_s(z^{m} S^{*m}(f))   \|_{\apa}\nonumber \\[0.1cm]
\leq   \sum_{n=0}^{m-1}\frac{\left[F\left( \frac{p(n+1)}{2},\frac{p(n+1)}{2},\alpha+2;(1-s)^2 \right) \right]^{\frac{1}{p}}}{\left[(\alpha+1)B\left(\frac{np}{2}+1,~\alpha+1 \right)\right]^{\frac{1}{p}}}(1-s)^{b}s^{n} \left \| f \right \|_{\apa}\nonumber \\[0.1cm]
&+s^{\frac{\alpha+2}{p}-1}(1-s)^{b-\frac{\alpha+2}{p}}\left \{ \frac{(\alpha+1)2^{2\alpha+7-(m+1)p}}{9[(m+1)p-\alpha-2]}+2^{\alpha}\max\left \{ 1,2^{\alpha+4-(m+1)p} \right \}  \right \}^{\frac{1}{p}}  \| S^{*}  \|_{\apa\rightarrow\apa}^{m}  \| f  \|_{\apa}.
\end{align*}
By (\ref{eq1}), we get the desired result.    The proof is complete.

\end{proof}

  \begin{remark}
When \(b = 0\) and \(\alpha \geq 0\), the condition \(\frac{\alpha+2}{b+1} < p \leq 2\) cannot hold. The parameter \(b\) induces several essential distinctions between the generalized Hilbert operator $\mathcal{H}_b$ and the classical Hilbert operator $\mathcal{H}$.
\end{remark}

\begin{theorem}\label{thm2}
  Let $\alpha>-1$, $b \geq0$, $1< p< \infty$ such that
 $ \alpha+2<p(b+1)$. Then $$\|\mathcal{H}_b\|_{\apa\rightarrow\apa}\geq B\Big(\frac{\alpha+2}{p},~b+1-\frac{\alpha+2}{p}\Big).$$
\end{theorem}

\begin{proof}   Let $1<\gamma <\min\left \{ \alpha+2,p \right \} <p(b+1)$ and
  $$f_\gamma(z)=(1-z)^{-\frac{\gamma}{p}}, ~~~z\in\D.   $$
 It is easy to see that
  \begin{align*}
    \|f_\gamma\|^p_{\apa}=F\left(\frac{\gamma}{2},\frac{\gamma}{2},\alpha+2;1\right)
\end{align*}
and $\|f_\gamma\|_{\apa}<\infty$ if $\gamma<\alpha+2$.   Moreover, we  have $\lim\limits_{\gamma\rightarrow\alpha+2}\|f_\gamma\|_{\apa}=\infty$ (see \cite{Bo}).

By (\ref{eq8}), we get
\begin{align*}
  \mathcal{H}_b(f_\gamma)(z)&=\int_0^1\frac{(1-t)^{b-\frac{\gamma}{p}}}{(1-tz)^{b+1}}\dd t=B\left(1,~b+1-\frac{\gamma}{p}\right)F\left(1,b+1,2+b-\frac{\gamma}{p};z\right),
\end{align*}
i.e.,
\begin{align*}
\mathcal{H}_b(f_\gamma)(z)&=\frac{\Gamma\left(b+1-\frac{\gamma}{p}\right)\Gamma\left(\frac{\gamma}{p}\right)}{\Gamma(b+1)}\sum^\infty_{k=0}\frac{\Gamma(k+1)
\Gamma\left(k+b+1\right)}{\Gamma\left(k+b+2-\frac{\gamma}{p}\right)
\Gamma\left(k+\frac{\gamma}{p}\right)}\frac{\Gamma\left(k+\frac{\gamma}{p}\right)}{\Gamma\left(\frac{\gamma}{p}\right)}
\frac{z^k}{k!}.
\end{align*}
By Stirling's formula we get
\begin{align*}
  \frac{\Gamma(k+1)\Gamma\left(k+b+1\right)}{\Gamma\left(k+b+2-\frac{\gamma}{p}\right)  \Gamma\left(k+\frac{\gamma}{p}\right)}=1+O\left(\frac{1}{k+1}\right),
\end{align*}
which implies that
  \begin{align*}
\mathcal{H}_b(f_\gamma)(z)&=B\left(b+1-\frac{\gamma}{p},~\frac{\gamma}{p}\right)\sum^\infty_{k=0}\left(1+O\left(\frac{1}{k+1}\right)\right)
\frac{\Gamma\left(k+\frac{\gamma}{p}\right)}{\Gamma\left(\frac{\gamma}{p}\right)}\frac{z^k}{k!}\\[0.1cm]
&=B\left(b+1-\frac{\gamma}{p},~\frac{\gamma}{p}\right)\left(f_\gamma(z)+g_\gamma(z)\right),
\end{align*}
with
\begin{align*}
 \|g_\gamma\|_\infty\leq\frac{C}{\Gamma\left(\frac{\gamma}{p}\right)} \sum^\infty_{k=0}\frac{\Gamma\left(k+\frac{\gamma}{p}\right)}{(k+1)!}\leq C_{p,\alpha}<\infty.
\end{align*}
See \cite{Bo} for more details. Hence,
\begin{align*}
  \sup\limits_{1<\gamma<\alpha+2}\|g_\gamma\|_{\apa}\leq  \|g_\gamma\|_\infty \leq   C_{p,\alpha}<\infty.
\end{align*}
So,
\begin{align*}
  \|\mathcal{H}_b\|_{\apa\rightarrow\apa}\geq\frac{\|\mathcal{H}_bf_\gamma\|_{\apa}}{\|f_\gamma\|_{\apa}}\geq B\Big(b+1-\frac{\gamma}{p},~\frac{\gamma}{p}\Big)\frac{\|f_\gamma\|_{\apa}-\|g_\gamma\|_{\apa}}{\|f_\gamma\|_{\apa}}.
\end{align*}
Letting $\gamma\rightarrow\alpha+2$, it follows that
\begin{align*}
  \|\mathcal{H}_b\|_{\apa\rightarrow\apa}\geq B\Big(\frac{\alpha+2}{p},~b+1-\frac{\alpha+2}{p}\Big).
\end{align*}
The proof is complete.
\end{proof}

  \begin{remark}
 When \( b = 0 \), Theorems \ref{thm3} and \ref{thm2} respectively reduce to the classical results for the Hilbert operator \( \mathcal{H} \) on $\apa$, as discussed in  \cite{Bo}.
\end{remark}

   By combining Theorems \ref{thm3} and  \ref{thm2}, we have the following result.

\begin{corollary}\label{cor1}
  Let $b, \alpha\geq0$, $1< p< \infty$ such that  $p\geq2(\alpha+2)$. Then
  \begin{align*}
\|\mathcal{H}_b\|_{\apa\rightarrow\apa}= B\Big(\frac{\alpha+2}{p},~b+1-\frac{\alpha+2}{p}\Big).
  \end{align*}
  \end{corollary}

In Corollary \ref{cor1}, we provided the norm characterization of the generalized Hilbert operator $\mathcal{H}_b$ when $ p \geq2(\alpha+2)$. Next, we will investigate the norm characterization of the generalized Hilbert operator when  $p<2(\alpha+2)$.
The following conclusions indicate that this situation is very complicated.

\begin{theorem}\label{thm3.1}  Let $b, \alpha\geq0$ and $1< p< \infty$.   If one of the following three conditions holds,
  \begin{enumerate}
    \item[{\bf(i)}] $p_\alpha\leq p<2(\alpha+2)$ and $b\leq\frac{\alpha + 2}{p}$;
    \item [{\bf(ii)}] If $p \geq p_{\alpha}$, $2p - 3\alpha - 8 \geq 0$ and $b>\frac{\alpha + 2}{p}$;
    \item[{\bf(iii)}] $\frac{\alpha+2}{b+1}<p<p_\alpha$ and $b<\frac{\alpha + 2}{p}$ and
    \begin{align}\label{eq:3.1}
      &\int_0^1\psi_{b,p,\alpha}(s)\int_{s^2}^1(1-r)^\alpha r^{p-\alpha-3}\dd r\dd s \leq B\left(1+\alpha,~2+\frac{\alpha}{2}\right)B\Big(\frac{\alpha+2}{p},~b+1-\frac{\alpha+2}{p}\Big),
    \end{align}
  \end{enumerate}
 then
  \begin{align*}
\|\mathcal{H}_b\|_{\apa\rightarrow\apa}=B\Big(\frac{\alpha+2}{p},~b+1-\frac{\alpha+2}{p}\Big).
\end{align*}
\end{theorem}

\begin{proof}   Let $f\in\apa$. From Theorem \ref{thm2} we only need to prove that
    \begin{align}\label{eq304-2}
\|\mathcal{H}_b\|_{\apa\rightarrow\apa}\leq B\Big(\frac{\alpha+2}{p},~b+1-\frac{\alpha+2}{p}\Big).
\end{align}
Combining (\ref{eq-var}), (\ref{eq126}), (\ref{eq-694}),  (\ref{eq919})  and (\ref{eq305-1})  we have
 \begin{align} \label{eq304-4}
    \|T_s(f)\|^p_{\apa}
    &=\frac{\alpha+1}{s^{p-2}(1-s)^{2-pb}}\int_{\phi_s(\D)}|z|^{p-4}|f(z)|^p(1-|\phi_s^{-1}(z)|^2)^\alpha\daz\nonumber\\[0.1cm]
  %  &=\frac{\alpha+1}{s^{p-2}(1-s)^{2-pb}}\int_{D_s}|z|^{p-4}|f(z)|^p    \left(\frac{s}{1-s}\frac{1}{|z|^2}\frac{\rho_s^2-|z-c_s|^2}{\rho_s}\right)^\alpha\daz\nonumber\\[0.1cm]
    %&=\frac{\alpha+1}{s^{p-2-\alpha}(1-s)^{2-pb+\alpha}}\int_{D_s}|z|^{p-4-2\alpha}|f(z)|^p    \left(\frac{\rho_s^2-|z-c_s|^2}{\rho_s}\right)^\alpha\daz\nonumber\\[0.1cm]
    &\leq\frac{\alpha+1}{s^{p-2-\alpha}(1-s)^{2-pb+\alpha}}2^\alpha\int_{D_s}|z|^{p-4-\alpha}|f(z)|^p
    \left(1-|z|\right)^\alpha\daz\nonumber\\[0.1cm]
    &\leq\frac{\alpha+1}{s^{p-2-\alpha}(1-s)^{2-pb+\alpha}}2^\alpha\int_{R_{s^2}}|z|^{p-4-\alpha}|f(z)|^p
    \left(1-|z|\right)^\alpha\daz\nonumber\\[0.1cm]
    &=\frac{\alpha+1}{s^{p-2-\alpha}(1-s)^{2-pb+\alpha}}2^\alpha\int_{s^2}^1\frac{2}{2\pi}\int_{0}^{2\pi} |f(re^{i\theta})|^p
    \dd \theta \left(1-r\right)^\alpha r^{p-3-\alpha}\dd r\nonumber\\[0.1cm]
     &=\frac{\alpha+1}{s^{p-2-\alpha}(1-s)^{2-pb+\alpha}}2^\alpha
     \int_{s^2}^1 \varphi(r)\left(1-r\right)^\alpha r^{p-3-\alpha}\dd r,
  \end{align}
  which implies that
  \begin{align}\label{eq306-1}
    \|T_s(f)\|_{\apa}
    &\leq \psi_{b,p,\alpha}(s)\left[2^\alpha(\alpha+1) \int_{s^2}^1\varphi(r)\left(1-r\right)^\alpha r^{p-3-\alpha}\dd r\right]^\frac{1}{p}.
  \end{align}
  Here $\varphi(r)=2M^p_p(r,f)$.  Thus, by (\ref{eq1}) and (\ref{eq306-1}),
  \begin{align}\label{eq306-3}
    \|\mathcal{H}_b(f)\|_{\apa} \leq\int_0^1 \psi_{b,p,\alpha}(s)\left[2^\alpha(\alpha+1) \int_{s^2}^1\varphi(r)\left(1-r\right)^\alpha r^{p-3-\alpha}\dd r\right]^\frac{1}{p}\dd s.
  \end{align}
 Using the fact that $(1+r)^\alpha\geq 2^\alpha r^{\frac{\alpha}{2}}$, we obtain
  \begin{align}\label{eq306-2}
   &  B\Big(\frac{\alpha+2}{p},~b+1-\frac{\alpha+2}{p}\Big)\|f\|_{\apa} \nonumber \\[0.1cm]
  %= &\int_0^1 \psi_{b,p,\alpha}(s)\dd s\left[(\alpha+1)\int_{0}^1 \varphi(r)\left(1-r^2\right)^\alpha r\dd r\right]^\frac{1}{p}\nonumber\\[0.1cm]
   = &\int_0^1 \psi_{b,p,\alpha}(s)\dd s\left[(\alpha+1)\int_{0}^1 \varphi(r)\left(1-r\right)^\alpha\left(1+r\right)^\alpha r\dd r\right]^\frac{1}{p}\nonumber\\[0.1cm]
  \geq  & \int_0^1 \psi_{b,p,\alpha}(s)\dd s\left[2^\alpha(\alpha+1)\int_{0}^1 \varphi(r)\left(1-r\right)^\alpha r^{1+\frac{\alpha}{2}}\dd r\right]^\frac{1}{p}.
  \end{align}
  Set
  \begin{align*}
    I_{p,\alpha}(s)=\int_{s^2}^1\varphi(r)\left(1-r\right)^\alpha r^{p-3-\alpha}\dd r, ~~
    J_\alpha=\int_{0}^1 \varphi(r)\left(1-r\right)^\alpha r^{1+\frac{\alpha}{2}}\dd r.
  \end{align*}
 Comparing (\ref{eq306-3}) with (\ref{eq306-2}), we conclude that (\ref{eq304-2}) is true   if
   \begin{align}\label{eq306-4}
    \int_0^1 \psi_{b,p,\alpha}(s)\left(I_{p,\alpha}(s)^{1/p}-J_\alpha^{1/p}\right)\dd s\leq0.
  \end{align}
% Note the fact that  $$
%  x^\beta-y^\beta\leq\beta y^{\beta-1}(x-y),\,\,\text{for}\,\,x>0,\,y>0\,\,\text{and}\,\,\beta\in(0,1).
%  $$
 Thus, by (\ref{eqxy}),  (\ref{eq306-4}) is true if the following inequality
  \begin{align}\label{eq304-5}
   \int_0^1\psi_{b,p,\alpha}(s)\left(I_{p,\alpha}(s)-J_\alpha\right)\dd s\leq0
  \end{align}
is valid.

An application of Fubini's Theorem gives that \eqref{eq304-5} holds if and only if
     \begin{align}\label{eq307-66}
     &\int_0^1\varphi(r)\left(1-r\right)^\alpha r^{p-3-\alpha}\int_0^{\sqrt{r}}\psi_{b,p,\alpha}(s)\dd s\dd r -\int_0^1\psi_{b,p,\alpha}(s)\dd s\int_{0}^1 \varphi(r)\left(1-r\right)^\alpha r^{1+\frac{\alpha}{2}}\dd r\leq0.
  \end{align}
  Furthermore, to prove (\ref{eq307-66}), it suffices to prove
       \begin{align}\label{eq307-6}
     \int_0^1\varphi(r)h_{b,p,\alpha}(r)\dd r\leq0,
  \end{align}
 where
 \begin{align*}
 h_{b,p,\alpha}(r)&=r^{p-\alpha-3}\left(1-r\right)^\alpha\int_0^{\sqrt{r}}\psi_{b,p,\alpha}(s)\dd s-\left(1-r\right)^\alpha r^{1+\frac{\alpha}{2}}\int_0^1\psi_{b,p,\alpha}(s)\dd s\nonumber\\[0.1cm]
 &=r^{1+\frac{\alpha}{2}}\left(1-r\right)^\alpha H_{b,p,\alpha}(\sqrt{r})
 \end{align*}
 and
 \begin{align*}
H_{b,p,\alpha}(r)=r^{2p-3\alpha-8}\int_0^{r}\psi_{b,p,\alpha}(s)\dd s-\int_0^1\psi_{b,p,\alpha}(s)\dd s.
 \end{align*}

{\bf Case (i) } $p_\alpha\leq p<2(\alpha+2)$ and $b\leq\frac{\alpha + 2}{p}$. By using Lemma \ref{lem308-6} we have \( H_{b,p,\alpha}(r) \leq 0 \) and \( h_{b,p,\alpha}(r) \leq 0 \) for all \( r \in (0,1] \). Since \( \varphi(r) \) is nonnegative, it follows that (\ref{eq307-6}) holds for \( p \geq p_{\alpha} \) as desired.

{\bf Case~(ii)} If $p \geq p_{\alpha}$, $2p - 3\alpha - 8 \geq 0$ and $b>\frac{\alpha + 2}{p}$. By using Lemma \ref{lem308-6} we have \( H_{b,p,\alpha}(r) \leq 0 \) and \( h_{b,p,\alpha}(r) \leq 0 \) for all \( r \in (0,1] \). Since \( \varphi(r) \) is nonnegative, it follows that (\ref{eq307-6}) holds for \( p \geq p_{\alpha} \) as desired.

{\bf Case~(iii)} When \( \frac{\alpha + 2}{b+1} < p < p_{\alpha} \) and $b<\frac{\alpha + 2}{p}$. Since \( \varphi(r) \) is continuously differentiable and increasing on the interval \([0,1)\), by Lemmas \ref{lem308-5} and   \ref{lem308-6} we conclude that (\ref{eq307-6}) holds if%\tr{(delete $\varphi(r)$)}
\begin{align*}
& \int_0^1h_{b,p,\alpha}(r)\dd r\\[0.1cm]
= &\int_0^1\left(1-r\right)^\alpha r^{p-3-\alpha}\int_0^{\sqrt{r}}\psi_{b,p,\alpha}(s)\dd s\dd r -\int_0^1\psi_{b,p,\alpha}(s)\dd s\int_{0}^1 \left(1-r\right)^\alpha r^{1+\frac{\alpha}{2}}\dd r\nonumber\\[0.1cm]
 =&\int_0^1\psi_{b,p,\alpha}(s)\left[\int_{s^2}^1\left(1-r\right)^\alpha r^{p-3-\alpha}\dd r-\int_{0}^1 \left(1-r\right)^\alpha r^{1+\frac{\alpha}{2}}\dd r\right]\dd s\nonumber\\[0.1cm]
% =&\int_0^1\psi_{b,p,\alpha}(s)\int_{s^2}^1\left(1-r\right)^\alpha r^{p-3-\alpha}\dd r\dd s-\int_0^1\psi_{b,p,\alpha}(s)\dd s\int_{0}^1 \left(1-r\right)^\alpha r^{1+\frac{\alpha}{2}}\dd r\nonumber\\[0.1cm]
= &\int_0^1\psi_{b,p,\alpha}(s)\int_{s^2}^1\left(1-r\right)^\alpha r^{p-3-\alpha}\dd r\dd s -B \Big( 1+\alpha,~ 2+\frac{\alpha}{2} \Big) B \Big( \frac{\alpha+2}{p}, ~b+1-\frac{\alpha+2}{p} \Big) \\[0.1cm]
\leq& 0.
\end{align*}
 This completes the proof of the   theorem.
  \end{proof}

By specializing \(\alpha\) to certain values and using Theorem \ref{thm3.1}, we deduce the following sequence of corollaries.

\begin{corollary}   Let $3 < p < \frac{1}{4}(11 + \sqrt{97})$ and $0 \leq b<\frac{3}{p}$.  If
\begin{align} \label{cor-0615}
\frac{1}{(p-3)(p-2)} -   \frac{B\Big(2p-6,~b+1\Big)}{(p-3) B\Big(\frac{3}{p}, ~2p-6\Big)} + \frac{B\Big(2p-4,~b+1\Big)}{(p-2) B\Big(\frac{3}{p},~ 2p-4\Big)} \leq \frac{4}{35},
\end{align}
then
\begin{align*}
\|\mathcal{H}_b\|_{A_1^p\rightarrow A_1^p} = B\Big(\frac{3}{p},~b+1-\frac{3}{p}\Big).
\end{align*}
\end{corollary}

\begin{proof}
 Since $0<\frac{3}{p}<1$ and $2p-6>0$, by Lemma \ref{lem2.1} we have
  \begin{align*}
&\int_0^1 \psi_{b,p,1}(s) \Big( \int_{s^2}^1 r^{p-4}(1-r) dr \Big) \dd s \\[0.1cm]
= & \int_0^1 \psi_{b,p,1}(s) \left[ \frac{1}{(p-3)(p-2)} - \frac{s^{2p-6}}{p-3} + \frac{s^{2p-4}}{p-2} \right] \dd s  \\[0.1cm]
= & \ \frac{1}{(p-3)(p-2)} B\Big(\frac{3}{p}, ~b+1-\frac{3}{p}\Big) - \frac{1}{p-3} B\Big(\frac{3}{p} + 2p - 6,~ b+1 - \frac{3}{p}\Big) \\[0.1cm]
& + \frac{1}{p-2} B\Big(\frac{3}{p} + 2p - 4,~ b+1 - \frac{3}{p}\Big) \\[0.1cm]
= & \ \frac{1}{(p-3)(p-2)} B\Big(\frac{3}{p},~ b+1-\frac{3}{p}\Big)- \frac{1}{p-3} \frac{B\Big(\frac{3}{p},~ b+1-\frac{3}{p}\Big)B\Big(2p-6,~ b+1\Big)}{ B\Big(\frac{3}{p}, ~2p-6\Big)}\\[0.1cm]
 & +\frac{1}{p-2}\frac{B\Big(\frac{3}{p}, ~b+1-\frac{3}{p}\Big)B\Big(2p-4,~ b+1\Big)}{ B\Big(\frac{3}{p}, ~2p-4\Big)}\\[0.1cm]
% = & \ \left[ \frac{1}{(p-3)(p-2)} - \frac{B\Big(2p-6,~b+1\Big)}{(p-3) B\Big(\frac{3}{p},~2p-6\Big)} + \frac{B\Big(2p-4,~b+1\Big)}{(p-2) B\Big(\frac{3}{p},~ 2p-4\Big)} \right] B\Big(\frac{3}{p},~ 1 +b- \frac{3}{p}\Big)\\[0.1cm]
= & \Big[ \frac{1}{(p-3)(p-2)} - f_b(p) \Big] B\Big(\frac{3}{p}, ~1 +b- \frac{3}{p}\Big).
\end{align*}

On the other hand,
\begin{align*}
B\Big(2,~ \frac{5}{2}\Big) = \int_0^1 r^{\frac{3}{2}}(1-r) dr = \frac{4}{35}.
\end{align*}

An application of {\bf case~(iii)} of Theorem \ref{thm3.1} for $\alpha = 1$ yields the desired conclusion.
\end{proof}

\begin{remark}
The literature \cite{Da} verifies that condition \eqref{cor-0615} fails to hold under the setting $\alpha = 1$ and $p=4<\frac{1}{4}(11+\sqrt{97})$. Consequently, Condition \eqref{eq:3.1} in Theorem \ref{thm3.1} cannot be satisfied for   all $\alpha\geq0$ and $\alpha + 2 < p < p_{\alpha}$.
 \end{remark}

\begin{corollary}
  Let $2 < p < 2 + \sqrt{3}$ and $0 \leq b<\frac{ 2}{p}$.  If
\begin{align*}
\frac{1}{p-2} -  \frac{B\Big(2p-4,~b+1\Big)}{(p-2) B\Big(\frac{2}{p},~ 2p-4\Big)} \leq \frac{1}{2},
\end{align*}
then
\begin{align*}
\|\mathcal{H}_b\|_{A^p\rightarrow A^p} = B\Big(\frac{2}{p},~b+1-\frac{2}{p}\Big).
\end{align*}
\end{corollary}

\begin{proof}
 By Lemma \ref{lem2.1} we have
\begin{align*}
&\int_0^1 \psi_{b,p,0}(s) \Big( \int_{s^2}^1 r^{p-3}dr \Big) \dd s
=   \int_0^1 \psi_{b,p,0}(s) \Big( \frac{1}{p-2} - \frac{s^{2p-4}}{p-2} \Big) \dd s  \\[0.1cm]
= & \ \frac{1}{p-2} B\Big(\frac{2}{p},~ b+1-\frac{2}{p}\Big) -\frac{1}{p-2} B\Big(\frac{2}{p} + 2p - 4,~ b+1 - \frac{2}{p}\Big) \\[0.1cm]
= & \ \frac{1}{p-2} B\Big(\frac{2}{p},~ b+1-\frac{2}{p}\Big)- \frac{1}{p-2} \frac{B\Big(\frac{2}{p},~ b+1-\frac{2}{p}\Big)B\Big(2p-4,~ b+1\Big)}{ B\Big(\frac{2}{p},~ 2p-4\Big)}\\[0.1cm]
= & \ \Big( \frac{1}{p-2} - f_b(p) \Big) B\Big(\frac{2}{p}, ~1 +b- \frac{2}{p}\Big).
\end{align*}
Since $B\left(1, 2\right) = \int_0^1(1-r) dr = \frac{1}{2},$  an application of {\bf case~(iii)} of Theorem \ref{thm3.1} for $\alpha = 0$ yields the desired conclusion.
\end{proof}

By Lemma \ref{lem2.1}, we can further obtain the following corollary.

\begin{corollary}\label{cor3.07}
  Let $\alpha \geq0$ and $\alpha + 2 < p < p_{\alpha}$ and $0\leq b<\frac{\alpha + 2}{p}$. If
  \begin{align*}%\label{eq3.0012}
\frac{1}{p - \alpha - 2} - \frac{B\Big(2(p-2-\alpha),~b+ 1\Big)}{(p - \alpha - 2)B\Big(\frac{\alpha + 2}{p},~2(p - \alpha - 2)\Big)}\leq B\Big(1 + \alpha,~ 2 + \frac{\alpha}{2}\Big),
\end{align*}
then
\begin{align*}
\|\mathcal{H}_b\|_{A_{\alpha}^p\rightarrow A_{\alpha}^p} = B\Big(\frac{\alpha+2}{p},~b+1-\frac{\alpha+2}{p}\Big).
\end{align*}
\end{corollary}

\begin{proof}  It is obvious that the condition \eqref{eq:3.1} holds if
\begin{align}\label{eq318-000}
\int_0^1 \psi_{b,p,\alpha}(s) \int_{s^2}^1 r^{p-\alpha-3} \dd r \dd s \leq B\Big(1+\alpha, ~2+\frac{\alpha}{2}\Big) B\Big(\frac{\alpha+2}{p},~ b+1-\frac{\alpha+2}{p}\Big).
\end{align}
An application of Lemma \ref{lem2.1} gives that
\begin{align*}
&(p - \alpha - 2) \int_0^1 \psi_{b,p,\alpha}(s) \int_{s^2}^1 r^{p-\alpha-3} \dd r \dd s
=    \int_0^1 \psi_{b,p,\alpha}(s) \left[1 - s^{2(p-\alpha-2)}\right] \dd s \\[0.1cm]
= & B\Big(\frac{\alpha + 2}{p},~ b+1 - \frac{\alpha + 2}{p}\Big) - B\Big(\frac{\alpha + 2}{p} + 2(p - \alpha - 2),~ b+1 - \frac{\alpha + 2}{p}\Big) \\[0.1cm]
= & B\Big(\frac{\alpha + 2}{p},~ b+1 - \frac{\alpha + 2}{p}\Big) - \frac{B\Big(\frac{\alpha + 2}{p},~b+ 1 - \frac{\alpha + 2}{p}\Big)B\Big(2(p-2-\alpha),~b+ 1\Big)}{ B\Big(\frac{\alpha + 2}{p},~ 2(p - \alpha - 2)\Big)}.
\end{align*}
Therefore, condition (\ref{eq318-000}) holds if and only if
\begin{align*}
& \int_0^1 \psi_{b,p,\alpha}(s) \int_{s^2}^1 r^{p-\alpha-3} \dd r \dd s \\[0.1cm]
=& \frac{B\Big(\frac{\alpha + 2}{p},~ b+1 - \frac{\alpha + 2}{p}\Big)}{p - \alpha - 2} - \frac{B\Big(\frac{\alpha + 2}{p},~b+ 1 - \frac{\alpha + 2}{p}\Big)B\Big(2(p-2-\alpha),~b+ 1\Big)}{ (p - \alpha - 2)B\Big(\frac{\alpha + 2}{p},~ 2(p - \alpha - 2)\Big)}\\[0.1cm]
\leq& B\Big(1+\alpha,~2+\frac{\alpha}{2}\Big)B\Big(\frac{\alpha+2}{p},~b+1-\frac{\alpha+2}{p}\Big).
\end{align*}
The proof is complete.
\end{proof}

Using a similar method, we can also obtain  a somewhat weaker conclusion.

\begin{corollary}\label{cor3.7}
  Let $\alpha \geq0$ and $\alpha + 2 < p < p_{\alpha}$ and $0\leq b<\frac{\alpha + 2}{p}$. If
 \begin{align} \label{eq3.12}
 \frac{p (b+1)- \alpha - 2}{(p - \alpha - 2)^2} - \frac{p (b+1)- \alpha - 2}{ 2(p-2-\alpha)^3B\left(\frac{\alpha + 2}{p},~ 2(p - \alpha - 2)\right)}\leq B\left(1 + \alpha,~ 2 + \frac{\alpha}{2}\right),
 \end{align}
then
\begin{align*}
\|\mathcal{H}_b\|_{A_{\alpha}^p\rightarrow A_{\alpha}^p} = B\Big(\frac{\alpha+2}{p},~b+1-\frac{\alpha+2}{p}\Big).
\end{align*}
\end{corollary}

\begin{proof}
  It is obvious that the condition \eqref{eq:3.1} holds if
\begin{align}\label{eq327.0}
\int_0^1 \psi_{b,p,\alpha}(s) \int_{s^2}^1 r^{p-\alpha-3} \dd r \dd s \leq B\left(1+\alpha, ~2+\frac{\alpha}{2}\right) B\Big(\frac{\alpha+2}{p},~ b+1-\frac{\alpha+2}{p}\Big).
\end{align}
By Lemma \ref{lem2.1}, we have
\begin{align*}
  B\Big(\frac{\alpha+2}{p},~ b+1-\frac{\alpha+2}{p}\Big)&=\frac{B\Big(1,~ b+1-\frac{\alpha+2}{p}\Big)B\Big(\frac{\alpha+2}{p}, ~ 1-\frac{\alpha+2}{p}\Big)}{B\Big(1-\frac{\alpha+2}{p},~b+1\Big)}\\[0.1cm]
  &\geq \frac{B\Big(1,~ b+1-\frac{\alpha+2}{p}\Big)B\Big(\frac{\alpha+2}{p},~ 1-\frac{\alpha+2}{p}\Big)}{B\Big(1-\frac{\alpha+2}{p},~1\Big)}\\[0.1cm]
  &=B\left(\frac{\alpha+2}{p}, ~1-\frac{\alpha+2}{p}\right)\frac{p-2-\alpha}{p(b+1)-2-\alpha}.
\end{align*}
Since $\psi_{b,p,\alpha}(s)\leq\psi_{0,p,\alpha}(s),$ it easy to see that \eqref{eq327.0} holds if
\begin{align}\label{eq-318}
 &\int_0^1 \psi_{0,p,\alpha}(s) \int_{s^2}^1 r^{p-\alpha-3} \dd r \dd s\nonumber\\[0.1cm]
 \leq & B\Big(1+\alpha,~ 2+\frac{\alpha}{2}\Big)B\Big(\frac{\alpha+2}{p}, ~1-\frac{\alpha+2}{p}\Big)\frac{p-2-\alpha}{p(b+1)-2-\alpha}.
\end{align}
Again, an application of Lemma \ref{lem2.1} with $b=0$ gives that
\begin{align*}
&(p - \alpha - 2) \int_0^1 \psi_{0,p,\alpha}(s) \int_{s^2}^1 r^{p-\alpha-3} \dd r \dd s
=  \int_0^1 \psi_{0,p,\alpha}(s) \left[1 - s^{2(p-\alpha-2)}\right] \dd s \\[0.1cm]
= & B\Big(\frac{\alpha + 2}{p}, ~1 - \frac{\alpha + 2}{p}\Big) - B\Big(\frac{\alpha + 2}{p} + 2(p - \alpha - 2),~ 1 - \frac{\alpha + 2}{p}\Big) \\[0.1cm]
= & B\Big(\frac{\alpha + 2}{p}, ~1 - \frac{\alpha + 2}{p}\Big) - \frac{B\Big(\frac{\alpha + 2}{p},~ 1 - \frac{\alpha + 2}{p}\Big)}{ 2(p-2-\alpha)B\Big(\frac{\alpha + 2}{p},~ 2(p - \alpha - 2)\Big)}.
\end{align*}
Therefore, the condition (\ref{eq-318}) holds if and only if
\begin{align*}
 \int_0^1 \psi_{0,p,\alpha}(s) \int_{s^2}^1 r^{p-\alpha-3} \dd r \dd s
=&\frac{B\Big(\frac{\alpha + 2}{p},~ 1 - \frac{\alpha + 2}{p}\Big)}{p - \alpha - 2} - \frac{B\Big(\frac{\alpha + 2}{p},~1 - \frac{\alpha + 2}{p}\Big)}{ 2(p - \alpha - 2)^2B\Big(\frac{\alpha + 2}{p},~ 2(p - \alpha - 2)\Big)}\\[0.1cm]
\leq& B\Big(1+\alpha,~ 2+\frac{\alpha}{2}\Big)B\Big(\frac{\alpha+2}{p}, ~1-\frac{\alpha+2}{p}\Big)\frac{p-2-\alpha}{p(b+1)-2-\alpha}.
\end{align*}
The proof is complete.
\end{proof}

The following theorem illustrates that for some small values of \( p \), the conjecture still holds.
Its proof mainly relies on the combination of Corollary \ref{cor3.7} and Lemma \ref{lem2.2}.

\begin{theorem}\label{thm3.8}
  Let $\alpha \geq0$  and $0\leq b<\frac{\alpha + 2}{p}$. Then there exists some $\beta_\alpha$ only depending on $\alpha$ with $\alpha + 2 < \beta_\alpha \leq \alpha + \frac{5}{2}$ such that
\begin{align*}
\|\mathcal{H}_b\|_{A_\alpha^p\rightarrow A_\alpha^p} &= B\Big(\frac{\alpha + 2}{p},~ b+1 - \frac{\alpha + 2}{p}\Big)
\end{align*}
for all $\alpha + 2 < p \leq \beta_\alpha$.
\end{theorem}

\begin{proof}
   Since $0 < \frac{\alpha + 2}{p} < 1, \  0 < 2(p - \alpha - 2) \leq 1,$ applying Lemma \ref{lem2.2} we obtain
\begin{align}\label{eq3.13}
&\frac{p(b+1)-\alpha-2}{(p-\alpha-2)^2} - \frac{p(b+1)-\alpha-2}{2(p-2-\alpha)^3 B\left(\frac{\alpha+2}{p}, ~2(p-\alpha-2)\right)} \nonumber\\[0.1cm]
\leq& \frac{p(b+1)-\alpha-2}{(p-\alpha-2)^2} - \frac{p(b+1)-\alpha-2}{2(p-2-\alpha)^3}\cdot   \frac{2(p-2-\alpha)\frac{\alpha+2}{p}\left[1 + 2(p-2-\alpha)\frac{\alpha+2}{p}\right]}{\frac{\alpha+2}{p} + 2(p-2-\alpha)} \nonumber\\[0.1cm]
= &\frac{p(b+1)-\alpha-2}{(p-\alpha-2)^2} - \frac{\frac{\alpha+2}{p}\left[1 + 2(p-2-\alpha)\frac{\alpha+2}{p}\right]\left[p(b+1)-\alpha-2\right]}{(p-2-\alpha)^2\left[\frac{\alpha+2}{p} + 2(p-2-\alpha)\right]} \nonumber\\[0.1cm]
%= &\frac{p(b+1)-\alpha-2}{(p-\alpha-2)^2}\Big[\frac{\frac{\alpha+2}{p} + 2(p-2-\alpha)}{\frac{\alpha+2}{p} + 2(p-2-\alpha)} - \frac{\frac{\alpha+2}{p}+2(p-2-\alpha)\Big(\frac{\alpha+2}{p}\Big)^2}{\frac{\alpha+2}{p} + 2(p-2-\alpha)}\Big] \nonumber\\[0.1cm]
%= &\frac{p(b+1)-\alpha-2}{(p-\alpha-2)^2}\left\{\frac{2(p-2-\alpha)\left[1 - \left(\frac{\alpha+2}{p}\right)^2\right]}{\frac{\alpha+2}{p} + 2(p-2-\alpha)}\right\}\nonumber \\[0.1cm]
= &\frac{p(b+1)-\alpha-2}{p-\alpha-2} \cdot \frac{2\Big[1 - \Big(\frac{\alpha+2}{p}\Big)^2\Big]}{\frac{\alpha+2}{p} + 2(p-2-\alpha)}\nonumber\\[0.1cm]
:=&\frac{p(b+1)-\alpha-2}{p-\alpha-2}f_\alpha(p).
\end{align}
As demonstrated in the proof of Theorem 3.8 in \cite{Da}, it can be established that $f_\alpha'(p) > 0$ and $f_\alpha(p)$ is increasing on $(\alpha+2, \frac{5}{2} + \alpha]$. Therefore,
\begin{align}\label{eq3.14}
\frac{p(b+1)-\alpha-2}{p-\alpha-2}f_\alpha(p) &\leq \frac{p(b+1)-\alpha-2}{p-\alpha-2}f_\alpha\left(\frac{5}{2} + \alpha\right) \nonumber\\[0.1cm]
&=\frac{2(p(b+1)-\alpha-2)}{(p-\alpha-2)(5 + 2\alpha)}.
\end{align}
If
\begin{equation*}%\label{eq3.15}
\frac{2(p(b+1)-\alpha-2)}{(p-\alpha-2)(5 + 2\alpha)}\leq B\Big(1 + \alpha,~ 2 + \frac{\alpha}{2}\Big),
\end{equation*}
from (\ref{eq3.13}) and (\ref{eq3.14}) it follows that (\ref{eq3.12}) holds for $\alpha + 2 < p \leq \alpha + \frac{5}{2}$.

On the other hand,
\begin{equation*}
\frac{2(p(b+1)-\alpha-2)}{(p-\alpha-2)(5 + 2\alpha)}> B\Big(1 + \alpha,~ 2 + \frac{\alpha}{2}\Big),
\end{equation*}
since $f_\alpha(\alpha+2) = 0$ and $f_\alpha(p)$ is increasing, by the intermediate value theorem there exists a unique constant $\beta_\alpha$ with $\alpha + 2 < \beta_\alpha < \alpha + \frac{5}{2}$ such that
\begin{equation*}
\frac{p(b+1)-\alpha-2}{p-\alpha-2}f(\beta_\alpha) = B\Big(1 + \alpha,~ 2 + \frac{\alpha}{2}\Big).
\end{equation*}
In this case (\ref{eq3.12}) remains valid for $\alpha + 2 < p \leq \beta_\alpha$. We note that
\begin{equation*}
p_\alpha > 2 + \frac{3\alpha}{4} + \frac{3\alpha + 6}{4} = \frac{3\alpha}{2} + \frac{7}{2} > \alpha + \frac{5}{2} \geq \beta_\alpha.
\end{equation*}
Hence the desired result follows from Corollary \ref{cor3.7}.
\end{proof}

When $\alpha=1$, Lemma \ref{lem5.1} can be used to slightly weaken Condition $\it {(iii)}$ of Theorem \ref{thm3.1}.

\begin{proposition}\label{pro5.1} Let $b\geq 0$ and  $\frac{3}{b+1} < p < \min\{ \frac{3}{b}, \frac{1}{4}(11 + \sqrt{97})\}$. If
\begin{align*}
&\int_0^1 \psi_{b,p,1}(s) \left( \int_{s^2}^1 r^{p-4}(1-r) \dd r \right) \dd s - \int_0^1 \frac{s \psi_{b,p,1}(s)}{1-s} \left( \int_{s^2}^1 r^{p-\frac{9}{2}}(1-r)(1-r^{\frac{1}{2}}) \dd r \right) \dd s \\[0.1cm]
\leq& \frac{4}{35} B\left(\frac{3}{p},~ 1 +b- \frac{3}{p}\right),
\end{align*}
then
\begin{align*}
\|\mathcal{H}_b\|_{A_1^p\rightarrow A_1^p} = B\left(\frac{3}{p},~b+1-\frac{3}{p}\right).
\end{align*}
\end{proposition}

\begin{proof} Recall that
  \begin{align*}
    \frac{\rho_s^2-|z-c_s|^2}{|z|^2\rho_s}&\leq \frac{2|z| - s - (2-s)|z|^2}{(1-s)|z|^2} = \frac{(1-s)(1-|z|^2) - (1-|z|)^2}{(1-s)|z|^2}.
  \end{align*}
  Combining this with   \eqref{eq6006}, \eqref{eq-694} and \eqref{eq304-4}, we   obtain
  \begin{align*}%\label{eq318-0}
    \|T_s(f)\|_{A_1^p}
    &=\psi_{b,p,1}(s)\left[2\int_{R_{s^2}}|z|^{p-4}|f(z)|^p\left(\frac{\rho_s^2-|z-c_s|^2}{|z|^2\rho_s}\right)\daz\right]^{1/p}\nonumber\\[0.1cm]
    &\leq\psi_{b,p,1}(s)\left[2\int_{R_{s^2}}|z|^{p-6}|f(z)|^p\left(1-|z|^2-\frac{ (1-|z|)^2}{1-s}\right)\daz\right]^{1/p}.
  \end{align*}
  Therefore, by the proof of Theorem \ref{thm3.1}, we just need to prove that
  \begin{align}\label{eq318-2}
\int_0^1 \psi_{b,p,1}(s) \left[ \int_{s^2}^1 \varphi(r) r^{p-5} \left( 1 - r^2 - \frac{(1-r)^2}{1-s} \right) \dd r - \int_0^1 \varphi(r) r (1 - r^2) \dd r \right] \dd s \leq 0.
\end{align}
From \cite{Da} we know that
\begin{align*}
&\int_{s^2}^1 \varphi(r) r^{p-5} \left( 1 - r^2 - \frac{(1-r)^2}{1-s} \right) \dd r - \int_0^1 \varphi(r) r (1 - r^2) \dd r \\[0.1cm]
\leq &\int_{s^2}^1 \varphi(r) r^{p-5} (1-r) \left( 2r - \frac{2s(1-r^{\frac{1}{2}})r^{\frac{1}{2}}}{1-s} \right) \dd r - 2 \int_0^1 \varphi(r) r^{\frac{3}{2}} (1-r)\dd r.
\end{align*}
Consequently, to establish inequality \eqref{eq318-2}, it suffices to prove
\begin{align*}%\label{eq318-2}
\int_0^1 \psi_{b,p,1}(s) \left[ \int_{s^2}^1 \varphi(r) r^{p-5} (1-r) \left( r - \frac{s(1-r^{\frac{1}{2}})r^{\frac{1}{2}}}{1-s} \right) \dd r -  \int_0^1 \varphi(r) r^{\frac{3}{2}} (1-r)\dd r \right] \dd s\leq0.
\end{align*}
According to Fubini's theorem, the last expression is equivalent to
\begin{align}\label{eq318-33}
\int_0^1 \varphi(r) f_{b,p}(r) \dd r \leq 0,
\end{align}
where
\begin{align*}
f_{b,p}(r) &= r^{p-4}(1-r) \int_0^{\sqrt{r}} \psi_{b,p,1}(s) \dd s - r^{p-\frac{9}{2}}(1-r)(1-r^{\frac{1}{2}}) \int_0^{\sqrt{r}} \frac{s \psi_{b,p,1}(s)}{1-s} \dd s\\[0.1cm]
&\quad - r^{\frac{3}{2}}(1-r) \int_0^1 \psi_{b,p,1}(s) \dd s \\[0.1cm]
&= r^{\frac{3}{2}}(1-r) g_{b,p}(\sqrt{r})
\end{align*}
and
\begin{align*}
g_{b,p}(r) &= r^{2p-11} \int_0^r \psi_{b,p,1}(s) \dd s - r^{2p-12}(1-r) \int_0^r \frac{s \psi_{b,p,1}(s)}{1-s} \dd s - \int_0^1 \psi_{b,p,1}(s) \dd s.
\end{align*}
Applying Lemmas \ref{lem308-5} and   \ref{lem5.1} we conclude that (\ref{eq318-33}) holds if
\begin{align*}
&\int_0^1 f_{b,p}(r) \, dr\\[0.1cm]
= &\int_0^1 \psi_{b,p,1}(s) \left\{ \int_{s^2}^1 r^{p-5}(1-r) \left[ r - \frac{s(1-r^{\frac{1}{2}})r^{\frac{1}{2}}}{1-s} \right] \, dr - \int_0^1 r^{\frac{3}{2}}(1-r) \, dr \right\} \, ds \\[0.1cm]
= &\int_0^1 \psi_{b,p,1}(s) \left[ \int_{s^2}^1 r^{p-4}(1-r) \, dr \right] \, ds - \int_0^1 \frac{s \psi_{b,p,1}(s)}{1-s} \left[ \int_{s^2}^1 r^{p-\frac{9}{2}}(1-r)(1-r^{\frac{1}{2}}) \, dr \right] \, ds \\[0.1cm]
&\quad - \frac{4}{35} B\Big( \frac{3}{p},~ 1 +b- \frac{3}{p} \Big) \\[0.1cm]
 \leq & 0.
\end{align*}
The proof is complete.
\end{proof}

When $\alpha=0$, Lemma \ref{l28} can also be used to slightly weaken Condition $\bf {(iii)}$ of Theorem \ref{thm3.1}.
Similarly to the  proof of Proposition \ref{pro5.1}, we get the following result.   We omit the details.

\begin{proposition}\label{pro5..11}   Let   $1< p< \infty$,  $ b \geq 0$ such that  $\frac{2}{b+1} < p < \min\{\frac{2}{b}, 2 + \sqrt{3}\}$. If
\begin{align*}
&\int_0^1 \psi_{b,p,0}(s) \left[ \int_{s^2}^1 r^{p-4}(1-r) \dd r \right] \dd s - \int_0^1 \frac{s \psi_{b,p,0}(s)}{1-s} \left[\int_{s^2}^1 r^{p-\frac{2}{p}}(1-r)(1-r^{\frac{1}{2}}) \dd r \right] \dd s \\[0.1cm]
\leq &\frac{1}{2} B\Big(\frac{2}{p},~ 1 +b- \frac{2}{p}\Big),
\end{align*}
then
\begin{align*}
\|\mathcal{H}_b\|_{A^p\rightarrow A^p} = B\Big(\frac{2}{p},~b+1-\frac{2}{p}\Big).
\end{align*}
\end{proposition}
\vskip 8mm

\section{$\|\mathcal{H}_b\|_{\apa\rightarrow \apa}$ for $-1<\alpha<0$}
\vskip 2mm

In this section, we will further   characterize    the norm of the generalized Hilbert operator $\mathcal{H}_b$ for the case where $-1<\alpha<0$. Using Lemma \ref{lem66},  the range of $\alpha$ in Theorem \ref{thm3} can be extend to $-1<\alpha<0$.

\begin{theorem}\label{thm7}
  Let $-1<\alpha<0$, $b\geq 0$, $1< p< \infty$ such that $ \alpha+2<p(b+1)  $. Then the following statements hold.
  \begin{enumerate}
		\item[\bf(i)] If $p\geq2(\alpha+2)$, then
$$
\|\mathcal{H}_b\|_{\apa\rightarrow\apa}\leq B\Big(\frac{\alpha+2}{p},~b+1-\frac{\alpha+2}{p}\Big).
$$
		\item[\bf(ii)] If $2\alpha+3\leq p<2(\alpha+2)$, then
$$
\|\mathcal{H}_b\|_{\apa\rightarrow\apa}\leq B\Big(\frac{\alpha+2}{p},~b+1-\frac{\alpha+2}{p}\Big)+2^{\frac{1}{p}}B\Big(\frac{\alpha+1}{p},~b+1-\frac{\alpha+1}{p}\Big).
$$
        \item[\bf(iii)] If $\alpha+2<p<2\alpha+3$, then
$$
\|\mathcal{H}_b\|_{\apa\rightarrow\apa}\leq B\Big(\frac{\alpha+2}{p},~b+1-\frac{\alpha+2}{p}\Big)+2^{\frac{2(\alpha+2)}{p}-1}B\Big(1-\frac{\alpha+2}{p},~b+\frac{\alpha+2}{p}\Big).
$$

 \item[\bf(iv)] If $b>0, \frac{\alpha+2}{b+1}<p\leq \alpha+2$, then
{\small \begin{align*}
&\|\mathcal{H}_b\|_{\apa\rightarrow\apa}\\[0.1cm]
\leq &\sum_{n=0}^{{\lceil b\rceil}-1}\left[ (\alpha+1)B\Big( \frac{np}{2}+1,~\alpha+1 \Big)  \right ]^{-\frac{1}{p}} \int_{0}^{1}s^{n}(1-s)^{b}\left [ F\Big ( \frac{p(n+1)}{2},\frac{p(n+1)}{2},\alpha+2;(1-s)^{2} \Big)  \right ]^{\frac{1}{p}}\dd s\nonumber \\[0.1cm]
&+B\Big( \frac{\alpha+2}{p},~b+1-\frac{\alpha+2}{p}\Big) \left \{ \frac{(\alpha+1)2^{4\alpha+11-2({\lceil b\rceil}+1)p}}{9[({\lceil b\rceil}+1)p-2\alpha-2]}+ 2^{4\alpha+8-2({\lceil b\rceil}+1)p}\right \}^{\frac{1}{p}}\left \| S^{*} \right \|_{\apa\rightarrow\apa}^{\lceil b\rceil}\nonumber \\[0.1cm]
&+B\Big( {\lceil b\rceil}+1-\frac{\alpha+2}{p},~b+\frac{\alpha+2}{p}-{\lceil b\rceil} \Big)\left \{ \frac{(\alpha+1)2^{6\alpha+15-3({\lceil b\rceil}+1)p}}{9[({\lceil b\rceil}+1)-2\alpha-2]} +2^{6\alpha+12-3({\lceil b\rceil}+1)p}\right \}^{\frac{1}{p}}\left \| S^{*} \right \|_{\apa\rightarrow\apa}^{\lceil b\rceil}
\end{align*}  }
\noindent for $p<\frac{2\alpha+4}{{\lceil b\rceil}+1}$ and
\begin{align*}
&\|\mathcal{H}_b\|_{\apa\rightarrow\apa}\\[0.1cm]
\leq &  \sum_{n=0}^{{\lceil b\rceil}-1}\left [ (\alpha+1)B\Big( \frac{np}{2}+1,~\alpha+1 \Big)  \right ]^{-\frac{1}{p}}\int_{0}^{1}s^{n}(1-s)^{b}\left [ F\Big ( \frac{p(n+1)}{2},\frac{p(n+1)}{2}, \alpha+2;(1-s)^{2} \Big )  \right ]^{\frac{1}{p}}\dd s\nonumber \\[0.1cm]
&+B\Big( \frac{\alpha+2}{p},~b+1-\frac{\alpha+2}{p} \Big) \left \| S^{*} \right \|_{\apa\rightarrow\apa}^{\lceil b\rceil}
\end{align*}
\noindent for $p\geq \frac{2\alpha+4}{{\lceil b\rceil}+1}$.
\end{enumerate}

\end{theorem}

\begin{proof} Set
  \begin{align*}
  \varphi_s(z)=\rho_s z+c_s, \,\,z\in\D, \,\,\,0<s<1.
 \end{align*}
Here $\rho_s$ and $c_s$ are defined as \eqref{eq126}. We have (see \cite{ka})
\begin{align}\label{eq2020}
\varphi_s(\D)=\phi_s(\D)=D\Big(\frac{1}{2-s},~~~~\,\frac{1-s}{2-s}\Big)\subset\D.
\end{align}
    By (\ref{eq6006}) and (\ref{eq919}), we get
  \begin{align}\label{eq-4.3}
    \|T_s(f)\|^p_{\apa}
    &=\frac{\alpha+1}{s^{p-2}(1-s)^{2-pb}}\int_{\phi_s(\D)}|z|^{p-4}|f(z)|^p(1-|\phi_s^{-1}(z)|^2)^\alpha\daz\nonumber\\[0.1cm]
    &=\frac{\alpha+1}{s^{p-2}(1-s)^{2-pb}}\int_{\phi_s(\D)}|z|^{p-4}|f(z)|^p\Big(\frac{s}{1-s}\frac{1}{|z|^2}    \frac{\rho_s^2-|z-c_s|^2}{\rho_s}\Big)^\alpha\daz\nonumber\\[0.1cm]
    &=\frac{\alpha+1}{s^{p-2-\alpha}(1-s)^{2-pb+\alpha}}\int_{\phi_s(\D)}|z|^{p-4-2\alpha}|f(z)|^p    \Big(\frac{\rho_s^2-|z-c_s|^2}{\rho_s}\Big)^\alpha\daz.
  \end{align}

Next, the proof is divided into four mutually independent subcases.

 {\bf Case~(i)~ $p\geq2(\alpha+2)$.}

  From   \eqref{eq-4.3}, one readily deduces that

  \begin{align*}
    \|T_s(f)\|^p_{\apa}
  %  &=\frac{\alpha+1}{s^{p-2-\alpha}(1-s)^{2-pb+\alpha}}\int_{\phi_s(\D)}|z|^{p-4-2\alpha}|f(z)|^p    \left(\frac{\rho_s^2-|z-c_s|^2}{\rho_s}\right)^\alpha\daz\nonumber\\[0.1cm]
    &\leq\frac{\alpha+1}{s^{p-2-\alpha}(1-s)^{2-pb+\alpha}}\int_{\phi_s(\D)}
    |f(z)|^p\Big(\frac{\rho_s^2-|z-c_s|^2}{\rho_s}\Big)^\alpha\daz.
  \end{align*}
  Let $$z=\varphi_s(w)=\rho_s w+c_s,\,0<s<1.$$
   Then using (\ref{eq2020}) and Lemma \ref{lem66}  with $|\rho_s|+|c_s|=1$,  we get
    \begin{align*}
    \|T_s(f)\|^p_{\apa}&\leq\frac{\alpha+1}{s^{p-2-\alpha}(1-s)^{2-pb+\alpha}}\int_{ \phi_s(\D) }|f(z)|^p\Big(\frac{\rho_s^2-|z-c_s|^2}{\rho_s}\Big)^\alpha\daz\nonumber\\[0.1cm]
   % &=\frac{\alpha+1}{s^{p-2-\alpha}(1-s)^{2-pb+\alpha}}\rho_s^{\alpha}\int_{\D}|f(\varphi_s(w))|^p (1-|w|^2 )^\alpha |\varphi'_s(w)|^2\dd A(w)\nonumber\\[0.1cm]
    &=\frac{\alpha+1}{s^{p-2-\alpha}(1-s)^{2-pb+\alpha}}\rho_s^{\alpha+2}\int_{\D}|f(\varphi_s(w))|^p (1-|w|^2 )^\alpha \dd A(w)\nonumber\\[0.1cm]
    &= \frac{\rho_s^{\alpha+2}}{s^{p-2-\alpha}(1-s)^{2-pb+\alpha}} \|C_{\varphi _{s}}(f)   \|_{A_{\alpha}^{p}}^p\nonumber\\[0.1cm]
     &\leq s^{\alpha+2-p}(1-s)^{pb-\alpha-2}  \|f  \|_{A_{\alpha}^{p}}^p.
  \end{align*}
  Hence,
   \begin{align*}
    \|T_s(f)\|_{\apa}&\leq s^{\frac{\alpha+2}{p}-1}(1-s)^{b-\frac{\alpha+2}{p}}\|f\|_{\apa},
  \end{align*}
which together with (\ref{eq1}) implies that
  \begin{align*}
    \|\mathcal{H}_b(f)\|_{\apa}&\leq\int_0^1\|T_s(f)\|_{\apa}\dd s\leq \int_0^1s^{\frac{\alpha+2}{p}-1}(1-s)^{b-\frac{\alpha+2}{p}}\dd s\|f\|_{\apa}\nonumber\\[0.1cm]
    &=B\Big(\frac{\alpha+2}{p},~b+1-\frac{\alpha+2}{p}\Big)\|f\|_{\apa}.
  \end{align*}

  {\bf Case~(ii)~$2\alpha+3\le p<2(\alpha+2)$.}

    In this case, we have $\left | z \right |^{p-4-2\alpha}\le \left | z \right |^{-1}$ and $|z| \geq\frac{s}{2-s}$ for $z\in \phi_{s}(\D)$. Then by $\eqref{eq-4.3}$, we get
  \begin{align*}
    \|T_s(f)\|^p_{\apa}&=\frac{\alpha+1}{s^{p-2-\alpha}(1-s)^{2-pb+\alpha}}\int_{\phi_s(\D)}|z|^{p-4-2\alpha}|f(z)|^p
    \Big(\frac{\rho_s^2-|z-c_s|^2}{\rho_s}\Big)^\alpha\daz\nonumber\\[0.1cm]
   % &\leq\frac{\alpha+1}{s^{p-2-\alpha}(1-s)^{2-pb+\alpha}}\int_{\phi_s(\D)}|z|^{-1}|f(z)|^p    \Big(\frac{\rho_s^2-|z-c_s|^2}{\rho_s}\Big)^\alpha\daz\nonumber\\[0.1cm]
    &\leq\frac{\alpha+1}{s^{p-2-\alpha}(1-s)^{2-pb+\alpha}}\Big ( \frac{2-s}{s} \Big)\int_{\phi_s(\D)}|f(z)|^p\Big(\frac{\rho_s^2-|z-c_s|^2}{\rho_s}\Big)^\alpha\daz\nonumber\\[0.1cm]
  %  &=\frac{\alpha+1}{s^{p-2-\alpha}(1-s)^{2-pb+\alpha}}\left(1+\frac{2(1-s)}{s}\right)\int_{\phi_s(\D)}|f(z)|^p\Big(\frac{\rho_s^2-|z-c_s|^2}{\rho_s}\Big)^\alpha\daz\nonumber\\[0.1cm]
    &=\frac{\alpha+1}{s^{p-2-\alpha}(1-s)^{2-pb+\alpha}}\left [1+\frac{2(1-s)}{s}\right]\rho_{s}^{\alpha+2}\int_{\D}|f(\varphi_{s}(w))|^p(1- | w   |^{2})^\alpha\dd A(w)\nonumber\\[0.1cm]
    &=\left [\frac{1}{s^{p-2-\alpha}(1-s)^{2-pb+\alpha}}+\frac{2}{s^{p-1-\alpha}(1-s)^{1-pb+\alpha}}  \right]\rho_{s}^{\alpha+2}\left \|C_{\varphi _{s}}(f) \right \|_{A_{\alpha}^{p}}^p\nonumber\\[0.1cm]
    &\leq\left[\frac{1}{s^{p-2-\alpha}(1-s)^{2-pb+\alpha}}+\frac{2}{s^{p-1-\alpha}(1-s)^{1-pb+\alpha}}  \right]  \|f   \|^p_{A_{\alpha}^{p}}.
  \end{align*}
 Hence,
    \begin{align}\label{eqta2}
    \|T_s(f)\|_{\apa}&\leq \left [s^{\frac{\alpha+2}{p}-1}(1-s)^{b-\frac{\alpha+2}{p}}+2^\frac{1}{p}s^{\frac{\alpha+1}{p}-1}(1-s)^{b-\frac{\alpha+1}{p}}\right ]\|f\|_{\apa}.
  \end{align}
 It follows from (\ref{eq1}) and (\ref{eqta2}) that
     \begin{align*}
     \|\mathcal{H}_b(f)\|_{\apa}&\leq\int_0^1
    \left [s^{\frac{\alpha+2}{p}-1}(1-s)^{b-\frac{\alpha+2}{p}}+2^\frac{1}{p}s^{\frac{\alpha+1}{p}-1}(1-s)^{b-\frac{\alpha+1}{p}}\right ]\dd s\|f\|_{\apa}\nonumber\\[0.1cm]
     &=\left [ B\Big(\frac{\alpha+2}{p},~b+1-\frac{\alpha+2}{p}\Big)+2^\frac{1}{p}B\Big(\frac{\alpha+1}{p},~b+1-\frac{\alpha+1}{p}\Big) \right ] \|f\|_{\apa}.
  \end{align*}

   {\bf Case~(iii)~$\alpha+2< p<2\alpha+3$.}

    Let $$z=\varphi_s(w)=\rho_s w+c_s,\,0<s<1.$$
    From \eqref{eq-4.3}, since $\left | z \right |\ge \frac{s}{2-s}  $ for $z\in \phi_{s}(\D)$, we obtain
 \begin{align*}
    \|T_s(f)\|^p_{\apa}&=\frac{\alpha+1}{s^{p-2-\alpha}(1-s)^{2-pb+\alpha}}\int_{\phi_s(\D)}|z|^{p-4-2\alpha}|f(z)|^p
    \Big(\frac{\rho_s^2-|z-c_s|^2}{\rho_s}\Big)^\alpha\daz\nonumber\\[0.1cm]
    &\leq\frac{\alpha+1}{s^{p-2-\alpha}(1-s)^{2-pb+\alpha}}\left ( \frac{2-s}{s} \right )^{2\alpha+4-p}\int_{\phi_s(\D)}|f(z)|^p\Big(\frac{\rho_s^2-|z-c_s|^2}{\rho_s}\Big)^\alpha\daz\nonumber\\[0.1cm]
   % &=\frac{\alpha+1}{s^{\alpha+2}(1-s)^{2-pb+\alpha}}(2-s)^{2\alpha+4-p}\int_{\phi_s(\D)}|f(z)|^p\Big(\frac{\rho_s^2-|z-c_s|^2}{\rho_s}\Big)^\alpha\daz\nonumber\\[0.1cm]
    &=\frac{\alpha+1}{s^{\alpha+2}(1-s)^{2-pb+\alpha}}(2-s)^{2\alpha+4-p}\rho_{s}^{\alpha+2}\int_{\D}|f(\varphi_{s}(w))|^p(1- | w   |^{2})^\alpha \dd A(w)\nonumber\\[0.1cm]
    &=\frac{(2-s)^{2\alpha+4-p}}{s^{\alpha+2}(1-s)^{2-pb+\alpha}}\rho_{s}^{\alpha+2} \|C_{\varphi _{s}}(f)  \|_{A_{\alpha}^{p}}^p\nonumber\\[0.1cm]
    &\leq\frac{(2-s)^{2\alpha+4-p}}{s^{\alpha+2}(1-s)^{2-pb+\alpha}} \|f  \|^p_{A_{\alpha}^{p}},
  \end{align*}
which implies that
\begin{align*}
    \|T_s(f)\|_{\apa}&\leq \frac{(2-s)^{\frac{2(\alpha+2)}{p}-1}}{s^{\frac{\alpha+2}{p}}(1-s)^{\frac{\alpha+2}{p}-b}}\left \|f \right \|_{A_{\alpha}^{p}}.
  \end{align*}
  Notice that the condition $\alpha+2< p< 2\alpha+3$ implies $0<\frac{2(\alpha+2)}{p}-1<1$. Combining this inequality with  \eqref{eq7}, we arrive at the estimate
\begin{align*}
    \|T_s(f)\|_{\apa}&\leq \left[s^{\frac{\alpha+2}{p}-1}(1-s)^{b-\frac{\alpha+2}{p}}+2^{\frac{2(\alpha+2)}{p}-1}s^{-\frac{\alpha+2}{p}}(1-s)^{b+\frac{\alpha+2}{p}-1}\right]\|f\|_{\apa}.
\end{align*}
   Hence,
 \begin{align*}
     \|\mathcal{H}_b(f)\|_{\apa}&\leq\left [ B\Big(  \frac{\alpha+2}{p},~b+1-\frac{\alpha+2}{p}  \Big)+2^{\frac{2(\alpha+2)}{p}-1} B\Big( 1- \frac{\alpha+2}{p},~b+\frac{\alpha+2}{p} \Big )\right ]\|f\|_{\apa}.
  \end{align*}

  {\bf Case~(iv)~$b>0, \frac{\alpha+2}{b+1}< p\leq \alpha+2$.}

   For a fixed real number $b$, denote $m=\lceil b \rceil$ to simplify subsequent derivations. As in the case (v) of Theorem 3.1, for any $z\in \mathbb{D}$, we have
\begin{equation*}
f(z)=\sum_{n=0}^{m-1}\frac{f^{(n)}(0)}{n!}z^{n}+z^{m} S^{*m}(f)(z).
\end{equation*}
%where $S^{*m}$ stands for the $m$-fold iterate of the operator $S^{*}$.
% Taking the $\apa$-norm of $T_s(f)$ on both sides yields
% \begin{align*}
%\left \| T_s(f) \right \|_{\apa}\leq \sum_{n=0}^{m-1}\frac{1}{n!}\left \| T_s(z^{n}f^{(n)}(0)) \right \|_{\apa}+\left \| T_s(z^{m}\cdot S^{*m}(f)) \right \|_{\apa}.
%\end{align*}
By (\ref{eqtsm}) and making the change of variable of $ w=\varphi_s(z)=\rho_s z+c_s$, we obtain
\begin{align}
 &\left \| T_s(z^{m} S^{*m}(f)) \right \|_{\apa}^p\nonumber \\[0.1cm]
 %= & (\alpha+1)\int_{\mathbb{D}}\left | w_s(z) \right |^p \left |\phi_s(z) \right |^{mp}\left |S^{*m}(f)(\phi_s(z)) \right |^p( 1-\left |z \right |^2)^p\dd A(z)\nonumber \\[0.1cm]
= & \frac{\alpha+1}{s^{p-2}(1-s)^{2-bp}}\int_{\phi_s(\mathbb{D})}\left | w\right |^{(m+1)p-4}\left |S^{*m}(f)(w) \right |^p\left(1-\left| \phi_{s}^{-1}(w) \right|^{2} \right)^{\alpha}\dd A(w)\nonumber \\[0.1cm]
= & \frac{\alpha+1}{s^{p-2}(1-s)^{2-bp}}\int_{\phi_s(\mathbb{D})}\left | w\right |^{(m+1)p-4}\left |S^{*m}(f)(w) \right |^p\left (\frac{s}{1-s}\frac{1}{\left | w \right |^2 } \frac{\rho_s^2-|w-c_s|^2}{\rho_s} \right )^{\alpha}\dd A(w)\nonumber \\[0.1cm]
= & \frac{\alpha+1}{s^{p-\alpha-2}(1-s)^{\alpha+2-bp}}\int_{\phi_s(\mathbb{D})}\left | w\right |^{(m+1)p-2\alpha-4}\left |S^{*m}(f)(w) \right |^p\left( \frac{\rho_s^2-|w-c_s|^2}{\rho_s}\right)^{\alpha}\dd A(w)\nonumber \\[0.1cm]
= & \frac{\alpha+1}{s^{p-\alpha-2}(1-s)^{\alpha+2-bp}}\rho_s^{\alpha+2}\int_{\mathbb{D}}\left | \varphi_s(z)  \right |^{(m+1)p-2\alpha-4}\left | S^{*m}\left ( f\left ( \varphi_s(z) \right )  \right )  \right |^p(1-\left | z \right |^2 )^{\alpha}\dd A(z)\label{eqtsp}
%= & \frac{(\alpha+1)\rho_s^{\alpha+2}}{s^{p-\alpha-2}(1-s)^{\alpha+2-bp}}\int_{\mathbb{D}}\left | \frac{1+(1-s)z}{2-s} \right |^{(m+1)p-2\alpha-4}\left | S^{*m}\left ( f\left ( \varphi_s(z) \right )  \right )   \right |^p(1-\left | z \right |^2 )^{\alpha}\dd A(z).  \nonumber \\[0.1cm]
\end{align}

Next, we split the rest proof into two parts.

  First we consider the case $(m+1)p-2\alpha-4<0$.  Since
   $$\left | \varphi_s(z) \right |=\left | \frac{1+(1-s)z}{2-s} \right |\ge \frac{1-(1-s)\left | z \right | }{2-s}\ge \frac{s\left | z \right | }{2-s},$$
 we have
\begin{align}
&  \| T_s(z^{m}S^{*m}(f))  \|_{\apa}^p\nonumber \\[0.1cm]
\leq &\frac{(\alpha+1)(2-s)^{2\alpha+4-(m+1)p}\rho_s^{\alpha+2}}{s^{\alpha+2-mp}(1-s)^{\alpha+2-bp}}\int_{\mathbb{D}}\left | z \right |^{(m+1)p-2\alpha-4}\left | S^{*m}\left ( f\left ( \varphi_s(z) \right )  \right )  \right |^p(1-\left | z \right |^2 )^{\alpha}\dd A(z)\nonumber \\[0.1cm]
= &\frac{(\alpha+1)(2-s)^{2\alpha+4-(m+1)p}\rho_s^{\alpha+2}}{s^{\alpha+2-mp}(1-s)^{\alpha+2-bp}}  \Big[\int_{\left|z \right| \leq \frac{1}{2}}\left | z \right |^{(m+1)p-2\alpha-4}\left | S^{*m}\left ( f\left ( \varphi_s(z) \right )  \right )  \right |^p(1-\left | z \right |^2 )^{\alpha}\dd A(z)\nonumber \\[0.1cm]
&+  \int_{\frac{1}{2}< \left|z \right| <1}\left | z \right |^{(m+1)p-2\alpha-4}\left | S^{*m}\left ( f\left ( \varphi_s(z) \right )  \right )  \right |^p(1-\left | z \right |^2 )^{\alpha}\dd A(z)\Big]\nonumber \\[0.1cm]
:=  &\frac{(\alpha+1)(2-s)^{2\alpha+4-(m+1)p}\rho_s^{\alpha+2}}{s^{\alpha+2-mp}(1-s)^{\alpha+2-bp}}\left(I_3+I_4 \right).  \label{eqtsmz}
\end{align}
Since the operator $S^{*}$ is bounded on the weighted Bergman space and $p>\frac{\alpha+2}{b+1}\geq\frac{2+2\alpha}{m+1}$, we have
\begin{align}
I_3 &=\int_{\left|z \right| \leq \frac{1}{2}}\left | z \right |^{(m+1)p-2\alpha-4}\left | S^{*m}\left ( f\left ( \varphi_s(z) \right )  \right )  \right |^p(1-\left | z \right |^2 )^{\alpha}\dd A(z)\nonumber \\[0.1cm]
&\leq \int_{\left|z \right| \leq \frac{1}{2}}\left | z \right |^{(m+1)p-2\alpha-4}(1-\left | z \right |^2 )^{-2}\dd A(z)\left \| S^{*m}\left ( f\left ( \varphi_s \right )  \right )  \right \|_{\apa}^p\nonumber \\[0.1cm]
&\leq \frac{16}{9}\int_{\left|z \right| \leq \frac{1}{2}}\left | z \right |^{(m+1)p-2\alpha-4}\dd A(z)\left \| S^{*m}\left ( f\left ( \varphi_s \right )  \right )  \right \|_{\apa}^p\nonumber \\[0.1cm]
&\leq\frac{2^{2\alpha+7-(m+1)p}}{9[(m+1)p-2\alpha-2]}\left \| S^{*} \right \|_{\apa\rightarrow\apa}^{mp}\left \|  C_{\varphi _{s}}(f)   \right \|_{\apa}^p\label{eqi3}
\end{align}
and
\begin{align}
I_4&=\int_{\frac{1}{2}< \left|z \right| <1}\left | z \right |^{(m+1)p-2\alpha-4}\left | S^{*m}\left ( f\left ( \varphi_s(z) \right )  \right )  \right |^p(1-\left | z \right |^2 )^{\alpha}\dd A(z)\nonumber \\[0.1cm]
&\leq 2^{2\alpha+4-(m+1)p}\int_{\mathbb{D}}\left | S^{*m}\left ( f\left ( \varphi_s(z) \right )  \right )  \right |^p(1-\left | z \right |^2 )^{\alpha}\dd A(z)\nonumber \\[0.1cm]
%&=\frac{2^{2\alpha+4-(m+1)p}}{\alpha+1}\left \| S^{*m}\left ( f\left ( \varphi_s \right )  \right )  \right \|_{\apa}^p\nonumber \\[0.1cm]
&\leq \frac{2^{2\alpha+4-(m+1)p}}{\alpha+1}\left \| S^{*} \right \|_{\apa\rightarrow\apa}^{mp}\left \|  C_{\varphi _{s}}(f)    \right \|_{\apa}^p.\label{eqi4}
\end{align}
Substituting (\ref{eqi3}) with (\ref{eqi4}) into  (\ref{eqtsmz}) and then applying Lemma 4.1 with $|\rho_s|+|c_s|=1$, we obtain
\begin{align}
& \left \| T_s(z^{m} S^{*m}(f)) \right \|_{\apa}\nonumber \\[0.1cm]
 \leq &  s^{m-\frac{\alpha+2}{p}}(1-s)^{b-\frac{\alpha+2}{p}}(2-s)^{\frac{2\alpha+4}{p}-(m+1)}\rho_s^{\frac{\alpha+2}{p}}\nonumber \\[0.1cm]
&\times \left\{\frac{(\alpha+1)2^{2\alpha+7-(m+1)p}}{9[(m+1)p-2\alpha-2]}+2^{2\alpha+4-(m+1)p} \right\}^{\frac{1}{p}}\left \| S^{*} \right \|_{\apa\rightarrow\apa}^m \left \|  C_{\varphi _{s}}(f) \right \|_{\apa}\nonumber \\[0.1cm]
\leq &  s^{m-\frac{\alpha+2}{p}}(1-s)^{b-\frac{\alpha+2}{p}}(2-s)^{\frac{2\alpha+4}{p}-(m+1)} \nonumber \\[0.1cm]
&\times\left\{\frac{(\alpha+1)2^{2\alpha+7-(m+1)p}}{9[(m+1)p-2\alpha-2]}+2^{2\alpha+4-(m+1)p} \right\}^{\frac{1}{p}}\left \| S^{*} \right \|_{\apa\rightarrow\apa}^m \left \|  f \right \|_{\apa}\nonumber \\[0.1cm]
\leq  & s^{\frac{\alpha+2}{p}-1}(1-s)^{b-\frac{\alpha+2}{p}} \left\{\frac{(\alpha+1)2^{4\alpha+11-2(m+1)p}}{9[(m+1)p-2\alpha-2]}+2^{4\alpha+8-2(m+1)p} \right\}^{\frac{1}{p}}\left \| S^{*} \right \|_{\apa\rightarrow\apa}^m \left \|  f \right \|_{\apa}\nonumber \\[0.1cm]
&+s^{m-\frac{\alpha+2}{p}}(1-s)^{b+\frac{\alpha+2}{p}-m-1} \left\{\frac{(\alpha+1)2^{6\alpha+15-3(m+1)p}}{9[(m+1)p-2\alpha-2]}+2^{6\alpha+12-3(m+1)p} \right\}^{\frac{1}{p}}\left \| S^{*} \right \|_{\apa\rightarrow\apa}^m \left \|  f   \right \|_{\apa}. \label{eqtsap}
\end{align}
In the last inequality, since $\frac{2\alpha+4}{p}-(m+1)>0$, we used the following inequality:
\begin{align}
   (2-s)^{\frac{2\alpha+4}{p}-(m+1)} =& (s+2(1-s))^{\frac{2\alpha+4}{p}-(m+1)}  \nonumber \\
  \leq & 2^{\frac{2\alpha+4}{p}-(m+1)}s^{\frac{2\alpha+4}{p}-(m+1)}+2^{\frac{4\alpha+8}{p}-2(m+1)} (1-s)^{\frac{2\alpha+4}{p}-(m+1)}. \nonumber
\end{align}
Using  (\ref{eqts}),  (\ref{eqtso}) and  (\ref{eqtsap}), we get
{\small \begin{align*}
&\left \| T_s(f) \right \|_{\apa}\\[0.1cm]
 %\leq &\sum_{n=0}^{m-1}\frac{1}{n!}\left \| T_s(z^{n}f^{(n)}(0)) \right \|_{\apa}+\left \| T_s(z^{m}\cdot S^{*m}(f)) \right \|_{\apa}\nonumber \\[0.1cm]
\leq &\sum_{n=0}^{m-1}\frac{(1-s)^{b}s^{n} \left \| f \right \|_{\apa}}{\left[(\alpha+1)B\left(\frac{np}{2}+1,\alpha+1 \right)\right]^{\frac{1}{p}}}\left[F\left( \frac{p(n+1)}{2},\frac{p(n+1)}{2},\alpha+2;(1-s)^2 \right)\right]^{\frac{1}{p}}\nonumber \\[0.1cm]
&+s^{\frac{\alpha+2}{p}-1}(1-s)^{b-\frac{\alpha+2}{p}} \left\{\frac{(\alpha+1)2^{4\alpha+11-2(m+1)p}}{9[(m+1)p-2\alpha-2]}+2^{4\alpha+8-2(m+1)p} \right\}^{\frac{1}{p}}\left \| S^{*} \right \|_{\apa\rightarrow\apa}^m \left \|  f    \right \|_{\apa}\nonumber \\[0.1cm]
&+ s^{m-\frac{\alpha+2}{p}}(1-s)^{b+\frac{\alpha+2}{p}-m-1} \left\{\frac{(\alpha+1)2^{6\alpha+15-3(m+1)p}}{9[(m+1)p-2\alpha-2]}+2^{6\alpha+12-3(m+1)p} \right\}^{\frac{1}{p}}\left \| S^{*} \right \|_{\apa\rightarrow\apa}^m \left \|  f   \right \|_{\apa}.
\end{align*} }
Combining the last inequality with  (\ref{eq1}), we can immediately get the desired result. \\

%{\small \begin{align*}
%& \|\mathcal{H}_b(f)\|_{\apa}\\[0.1cm]
% \leq & \sum_{n=0}^{m-1}\left [ (\alpha+1)B\left ( \frac{np}{2}+1,\alpha+1 \right )  \right ]^{-\frac{1}{p}}\times \nonumber \\[0.1cm]
%& \int_{0}^{1}s^{n}(1-s)^{b}\left [ F \left ( \frac{p(n+1)}{2},\frac{p(n+1)}{2},\alpha+2;(1-s)^{2} \right )  \right ]^{\frac{1}{p}}\dd s\left \|  f   \right \|_{\apa}\nonumber \\[0.1cm]
%&+B\left ( \frac{\alpha+2}{p},b+1-\frac{\alpha+2}{p}\right ) \left ( \frac{(\alpha+1)2^{4\alpha+11-2(m+1)p}}{9[(m+1)p-2\alpha-2]}+ 2^{4\alpha+8-2(m+1)p}\right )^{\frac{1}{p}}\left \| S^{*} \right \|_{\apa\rightarrow\apa}^p\left \|  f   \right \|_{\apa}\nonumber \\[0.1cm]
%&+B\left ( m+1-\frac{\alpha+2}{p},b+\frac{\alpha+2}{p}-m \right )\left ( \frac{(\alpha+1)2^{6\alpha+15-3(m+1)p}}{9[(m+1)-2\alpha-2]} +2^{6\alpha+12-3(m+1)p}\right )^{\frac{1}{p}}\left \| S^{*} \right \|_{\apa\rightarrow\apa}^m\left \|  f   \right \|_{\apa}.
%\end{align*}   }

  Next we consider the case $(m+1)p-2\alpha-4 \ge 0$.  Since $\left | \varphi_s(z) \right |^{(m+1)p-2\alpha-4}\leq 1,$ from (\ref{eqtsp}) and Lemma 4.1, we have
\begin{align}
  \| T_s(z^{m}  S^{*m}(f))   \|_{\apa}^p \leq & \frac{(\alpha+1) \rho_s^{\alpha+2}}{s^{p-\alpha-2}(1-s)^{\alpha+2-bp}}\int_{\mathbb{D}}\left | S^{*m}\left ( f\left ( \varphi_s(z) \right )  \right )  \right |^p(1-\left | z \right |^2 )^{\alpha}\dd A(z)\nonumber \\[0.1cm]
= & \frac{ \rho_s^{\alpha+2}}{s^{p-\alpha-2}(1-s)^{\alpha+2-bp}}\left \| S^{*m}\left ( f\left ( \varphi_s(z) \right )  \right )  \right \|_{\apa}^p\nonumber \\[0.1cm]
\leq & \frac{ 1}{s^{p-\alpha-2}(1-s)^{\alpha+2-bp}}\left \| S^{*} \right \|_{\apa\rightarrow\apa}^{mp}\left \|  f   \right \|_{\apa}^p. \label{eqtspp}
\end{align}
Consequently,
\begin{equation*}
\left \| T_s(z^{m} S^{*m}(f)) \right \|_{\apa}\leq  s^{\frac{\alpha+2}{p}-1}(1-s)^{b-\frac{\alpha+2}{p}}\left \| S^{*} \right \|_{\apa\rightarrow\apa}^{m}\left \|  f   \right \|_{\apa}.
\end{equation*}
Using (\ref{eqts}), (\ref{eqtso}) and (\ref{eqtspp}), we get
\begin{align*}
& \left \| T_s(f) \right \|_{\apa}\\[0.1cm]
%\leq &\sum_{n=0}^{m-1}\frac{1}{n!}\left \| T_s(z^{n}f^{(n)}(0)) \right \|_{\apa}+\left \| T_s(z^{m}  S^{*m}(f)) \right \|_{\apa}\nonumber \\[0.1cm]
\leq& \sum_{n=0}^{m-1}\frac{(1-s)^{b}s^{n} }{\left[(\alpha+1)B\left(\frac{np}{2}+1,\alpha+1 \right)\right]^{\frac{1}{p}}}\left[F\left( \frac{p(n+1)}{2},\frac{p(n+1)}{2},\alpha+2;(1-s)^2 \right)\right]^{\frac{1}{p}} \left\| f \right \|_{\apa} \nonumber \\[0.1cm]
&+ s^{\frac{\alpha+2}{p}-1}(1-s)^{b-\frac{\alpha+2}{p}}\left \| S^{*} \right \|_{\apa\rightarrow\apa}^{m}\left \|  f   \right \|_{\apa},
\end{align*}
which implies the desired result.  The proof is complete.
\end{proof}

In particular, we can obtain the exact value of the norm of the generalized Hilbert operator $\mathcal{H}_b$ on $\apa$ when $-1<\alpha<0$ and $p\geq2(\alpha+2)$.

\begin{corollary} \label{cor4.1}
  Let $-1<\alpha<0$, $b\geq 0$, $1< p< \infty$ such that  $p\geq2(\alpha+2)$. Then
\begin{align*}\label{eq66+}
\|\mathcal{H}_b\|_{\apa\rightarrow\apa}=B\Big(\frac{\alpha+2}{p},~b+1-\frac{\alpha+2}{p}\Big).
\end{align*}
\end{corollary}

{\bf Data Availability}  No data was used to support this study.
	
{\bf Conflicts of Interest}  The authors  declare that they have no conflicts of interest.

{\bf Acknowledgements} 	The corresponding author was supported by  NNSF of China (No. 12371131). The third author was supported by Sichuan Provincial Natural Science Foundation of China (No. 2026NSFSC0718) and Guizhou Education Department Youth Science and Technology Talent Growth Project (No. QianJiaoJi [2024] 159).


\begin{thebibliography}{WWW}

\bibitem{BTH} G. Bao, L. Tian and  H. Wulan, The norm of the Hilbert matrix operator on Bergman spaces, {\it Canad. Math. Bull.} (2026), 1--18.

\bibitem{bw} G. Bao and  H. Wulan,   Hankel matrices acting on Dirichlet spaces, {\it J. Math. Anal.   Appl.}  {\bf 409} (2014), no.1, 228--235.

\bibitem{Bh} B. Bhayo and J. S$\acute{a}$ndor, On the inequalities for Beta function, {\it Notes Number Theory Discrete Math.} {\bf21} (2015) 1--7.


\bibitem{BK} V. Bo\v{z}in and B. Karapetrovi\'{c}, Norm of the Hilbert matrix on Bergman spaces, {\it J. Funct. Anal.} \textbf{274} (2018), 525--543.


\bibitem{Da}  J. Dai, On the norm of the Hilbert matrix operator on weighted Bergman spaces, {\it J. Funct. Anal.}  {\bf287} (2024), 110587.


\bibitem{Di} E. Diamantoupoulos, Hilbert matrix on Bergman spaces, {\it Illinois J. Math.} \textbf{48} (2004), 1067--1078.

\bibitem{DS} E. Diamantopoulos and A. Siskakis, Composition operators and Hilbert matrix, {\it Studia Math.} \textbf{140} (2000), 191--198.

\bibitem{DJV} M. Dostanni\'{c}, M. Jevti\'{c} and D. Vukoti\'{c}, Norm of the Hilbert matrix on Bergman and Hardy spaces and a theorem
of Nehari type, \textit{J. Funct. Anal.} \textbf{254} (2008), 2800--2815.

%\bibitem{EMOT} A. Erd\'{e}lyi, W. Magnus, F. Oberhettinger and F. Tricomi, {\it Higer Transcendental Functions}, vol. I, McGraw-Hill, New York, 1973.

\bibitem{ggps}  P. Galanopoulos, D. Girela, J. Peláez and A. Siskakis, Generalized Hilbert operators,  {\it Ann. Acad. Sci. Fenn. Math.} {\bf 39} (2014), 231--258.


%\bibitem{HL}  G.   Hardy and J.  Littlewood, Some properties of fractional integrals. II, {\it Math. Z.} \textbf{34} (1932), no. 1, 403--439.


\bibitem{JK} M. Jevti\'{c} and B. Karapetrovi\'{c}, Hilbert matrix on spaces of Bergman-type, \textit{J. Math. Anal. Appl.} \textbf{453} (2017), 241--254.

\bibitem{JVA}  M. Jevti\'{c}, D. Vukoti\'{c} and M. Arsenovi\'{c},  {\it Taylor Coefficients and Coefficient Multipliers of Hardy and Bergman-Type Spaces}, RSME Springer Series, Springer, 2016.

\bibitem{Bo}  B. Karapetrovi\'{c}, Norm of the Hilbert matrix operator on the weighted Bergman spaces,  \textit{Glasgow Math. J.}  {\bf 60} (2018), 513--525.


\bibitem{ka} B. Karapetrovi\'{c}, Hilbert matrix and its norm on weighted Bergman spaces, {\it J. Geom. Anal.}    {\bf 31} (2021), 5909--5940.

 \bibitem{li} S. Li, Generalized Hilbert operator on Dirichlet type spaces, \textit{Appl. Math. Comput.} \textbf{214} (2009), 304--309.

\bibitem{ls}  S. Li and S. Stevi\'c, Generalized Hilbert operator and Fej\'er-Riesz type inequalities on the polydisc,  {\it Acta Math. Sci.} {\bf29} (2009), 191--200.

%\bibitem{LM2} M. Lindstr\"{o}m, S. Miihkinen and D. Norrbo, Exact essential norm of generalized Hilbert matrix operators on classical analytic function spaces,  {\it  Adv. Math.} \textbf{408} (2022),  1--34.

\bibitem{LM} M. Lindstr\"{o}m, S. Miihkinen and N. Wikman, On the exact value of the norm of the Hilbert matrix operator on the weighted Bergman spaces,  {\it Ann. Fenn. Math.} \textbf{46} (2021),  201--224.

\bibitem{liu} C. Liu, Sharp Forelli-Rudin estimates and the norm of the Bergman projection, {\it J. Funct. Anal.}  \textbf{268} (2015), no.2, 255--277.


\bibitem{wzz} H. Wulan, M. Zhou and J. Zhu, Hilbert matrix norms on weighted Bergman spaces: even exponents and a counterexample to the beta formula,
https://doi.org/10.48550/arXiv.2607.23540.



%\bibitem{Lu} Y. Luke,  {\it The Special Functions and Their Approximations}, New York, NY: Academic Press, 1969.
%
%\bibitem{ma} W. Magnus, On the spectrum of Hilbert's matrix,  {\it Amer.J. Math.} \textbf{72} (1950), 699--704.
%
%\bibitem{Wu} Z. Wu, Carleson measures and multipliers for Dirichlet spaces,  {\it J. Funct. Anal.}  \textbf{169} (1999), 148--163.

\bibitem{Zhu2} K. Zhu,  {\it Operator Theory in Function Spaces},   Second Edition, Amer. Math. Soc.  2007.

\end{thebibliography}
\end{document}